\documentclass[12pt]{article}
\usepackage{amsmath}
\usepackage{subfigure}
\usepackage{amsfonts}
\usepackage{float}
\usepackage{amsthm}
\usepackage{amssymb, setspace}
\usepackage{color,graphicx, epsfig, geometry, hyperref,fancyhdr}
\usepackage{mathrsfs}
\usepackage[T1]{fontenc}
\usepackage{latexsym,amssymb,amsmath,amsfonts,amsthm}
\usepackage{graphicx}
\usepackage{color}
\usepackage{soul}
\usepackage{caption}
\usepackage{dsfont}
\usepackage{multirow}
\usepackage{rotating}
\newtheorem{theorem}{Theorem}[section]

\newtheorem{corollary}{Corollary}[section]
\newtheorem{lemma}{Lemma}[section]
\newtheorem{definition}{Definition}[section]
\newtheorem{assumption}{Assumption}[section]

\newcommand{\dif}{\mathrm{d}}

\newcommand{\N}{\mathbb{N}}

\newcommand{\R}{\mathbb{R}}

\newcommand{\grad}{\nabla}

\newcommand{\ep}{\varepsilon}

\newcommand{\pa}{\partial}

\newcommand{\ov}{\overline}

\graphicspath{{pics/}}

\begin{document}
\setlength{\parskip}{1mm}
\setlength{\oddsidemargin}{0.1in}
\setlength{\evensidemargin}{0.1in}
\lhead{}
\rhead{}

{\bf \Large \noindent The interior transmission eigenvalue problem for an inhomogeneous media with a conductive boundary}

\begin{center}
O. Bondarenko\\
Department of Mathematics\\
Karlsruhe Institute of Technology\\
76131 Karlsruhe, Germany\\
\vspace{0.2in}
Isaac Harris\\
Department of Mathematics\\
Texas A $\&$ M University \\
College Station, Texas 77843-3368\\
E-mail: iharris@math.tamu.edu\\
\vspace{0.2in}
A. Kleefeld\\
Institute for Applied Mathematics and Scientific Computing\\
Brandenburg University of Technology\\
03046 Cottbus, Germany
\end{center}

\begin{abstract}
\noindent In this paper, we investigate the interior transmission eigenvalue problem for an inhomogeneous media with  conductive boundary conditions. We prove the discreteness and existence of the transmission eigenvalues. We also investigate the inverse spectral problem of gaining information about the material properties from the transmission eigenvalues. In particular, we prove that the first transmission eigenvalue is a monotonic function of the refractive index $n$ and boundary conductivity parameter $\eta$, and obtain a uniqueness result for constant coefficients. We provide some numerical examples to demonstrate the theoretical results in three dimensions. 
\end{abstract}


\section{Introduction}

The interior transmission eigenvalue problems have become an important area of research in 
inverse scattering theory. It has been shown that the transmission eigenvalues 
can be determined from the measured scattering data (see e.g. \cite{cchlsm}, \cite{cavityhcs}, \cite{armin}, and \cite{monkyang}). Furthermore, with the knowledge of the transmission eigenvalues it is possible to retrieve information about the material
properties of the scattering object. For example, \cite{cgtev} and \cite{gpinvtev} show that constant and piecewise constant refractive indices, respectively, can be reconstructed with the knowledge of the transmission eigenvalues.
In \cite{cavitycch},  \cite{cavitych}, and \cite{cavityhcs} the transmission eigenvalues are used to detect 
cavities (that is, the subregions in the scatterer where the contrast is zero). This suggests that the transmission eigenvalues can have practical applications, for example, in non-destructive testing.
%
%
%
%
%

In this paper, we study the interior transmission eigenvalue problem associated with the following 
scattering problem: Let $D \subset \R^m$, $m\in \{2,3\}$, be a collection of bounded simply connected domains with piecewise smooth boundary $\pa D$, and let $n$ denote the refractive index, $\nu$ the outward unit normal to the boundary $\pa D$,
$k> 0$ the wave number and $\eta$  a boundary parameter. The total field $u(x) = \mathrm{e}^{\mathrm{i} k {x} \cdot d}+ u^s(x)$ for $x \in \R^m$ and the incidence direction $d \in \mathbb{S} =
\{x \in \R^m: \, |x| = 1\}$ satisfies the following set of equations:
\begin{align}
\Delta u +k^2 u=0  \quad \textrm{ in }  \R^m \setminus \overline{D}  \quad  \text{and} \quad  \Delta u +k^2 nu=0  \quad \textrm{ in } \,  {D},  \label{direct1}\\
 u_+-u_-=0  \quad  \textrm{ on } \pa D \quad  \text{and} \quad  \frac{\partial u_+}{\partial \nu}+ \eta u_+= \frac{\partial u_-}{\partial \nu} \quad \textrm{ on } \partial D, \label{direct3}\\
\lim\limits_{r \rightarrow \infty} r^{\frac{m-1}{2}} \left( \frac{\partial u^s}{\partial r} -\mathrm{i}k u^s \right)=0. \label{src}
\end{align}
The radiation condition \eqref{src} is satisfied uniformly with respect to the direction $\hat{x}=x/|x|$. For the case where $\eta$ is (positive) purely imaginary the above problem represents a model for scattering by an inhomogeneous 
medium covered with a thin and highly conductive layer. We refer to \cite{fmconductbc}, where the authors studied 
the well-posedness of the direct problem and the inverse problem of reconstruction of the domain $D$ via the 
factorization method.
%

We assume that $D$ is given. The interior transmission eigenvalue problem corresponding to \eqref{direct1}--\eqref{src} is to determine the values of $k>0$ such that there exists a nontrivial solution to
\begin{align}
\Delta w +k^2 n w=0 \quad \text{and} \quad \Delta v + k^2 v=0  \quad &\textrm{ in } \,  D \label{teprob1} \\
 w-v=0  \quad  \text{and} \quad  \frac{\partial w}{\partial \nu}-\frac{\partial v}{\partial \nu}= \eta v \quad &\textrm{ on } \partial D.  \label{teprob2} 
\end{align} 
We will call such values of $k$ the \textit{interior transmission eigenvalues}. In this work we will consider the case where
$\eta$ is real valued and positive. We will show that this case can be treated in a similar way as the transmission eigenvalue problems considered in \cite{chtevexist} and \cite{monkyang}. In Section 2, we define 
the interior transmission eigenvalue problem in the appropriate Sobolev spaces. We then show in Section 3 that the eigenvalues form 
at most a discrete set with infinity as the only accumulation point and in Section 4 we show the existence. Using a version of the Courant-Fischer min-max principle we will obtain monotonicity results for the transmission eigenvalues with respect to the material parameters $n$ and $\eta$ in Section 5. Finally, in Section 6 we provide some numerical examples to demonstrate the theoretical results using a boundary integral formulation to compute the transmission eigenvalues. A short summary concludes this article.

\section{Problem definition and variational formulation}
\noindent Let $D \subset \R^m$, $m\in \{2,3\}$, represent a collection of bounded simply connected 
domains. 
We define the Sobolev space  
$$H^1_0(D)=\big\{ u\in L^2(D):\, | \grad u | \in L^2(D)\, \text{ and } \, u=0 \, \text{ on } \, \partial D \big\}$$
and 
$$\tilde{H}^2_0(D) = \big\{u\in H^2(D):\, u\in H^2(D)\cap H_0^1(D) \big\}.$$
By construction, the space $\tilde{H}^2_0(D)$ is a subspace 
of $H^2(D)$ equipped with the $H^2(D)$ norm defined as
\begin{align}
\|u\|_{H^2(D)} = \sum_{|\alpha| \leq 2} \|D^\alpha u\|_{L^2(D)},
\end{align}
$\alpha := (\alpha_1, \dots, \alpha_m),\, \,  \alpha_j \in \N_0, \, \, j = 1, \dots, m, \, \, |\alpha| = 
\alpha_1 + \dots + \alpha_m$. \\ \\

The interior transmission eigenvalue problem reads as follows: 
for given functions $n\in L^\infty(D)$ and $\eta \in L^\infty(\pa D)$ 
find $k>0$ and nontrivial $(w,v)\in L^2(D)\times L^2(D)$ such that 
$w-v \in \tilde{H}^2_0(D)$
and $(w,v)$ satisfies
\begin{align}
\label{eig_helm_1} \Delta w + k^2 n w &= 0 \text{ in } D, \\
\label{eig_helm_2}\Delta v + k^2 v &= 0 \text{ in } D, \\
\label{eig_bc1}\frac{\pa w}{\pa \nu} - \frac{\pa v}{\pa \nu} &= \eta v \text{ on } \pa D, \\
\label{eig_bc2} w - v &= 0 \text{ on } \pa D.
\end{align}
For analytical considerations we put the following assumptions on $n$, $\eta$, and $\pa D$.
\begin{assumption} \quad
\begin{enumerate}
\item The boundary $\partial D$ is of class $\mathcal{C}^2$.
\item $n$ is real-valued. It holds either $0<n_{\text{min}}\leq n<1$ or $n>1$ a.e. in $D$.
\item $\eta \in L^{\infty}(\partial D)$ is real-valued such that  $\eta>0$ a.e. on $\pa D$. 
\end{enumerate}
\end{assumption}
%

The pair $(w,v)\in L^2(D)\times L^2(D)$ is assumed to satisfy (\ref{eig_helm_1})--(\ref{eig_helm_2})
in the distributional sense. We now let $u \in \tilde{H}^2_0(D)$ denote the difference $w-v$. Therefore $u$ satisfies 
\begin{align}\label{eig_eq_v}
\Delta u + k^2 n u = - k^2(n-1)v \, \text{ in } D
\end{align}
or 
\begin{align}\label{eig_diff_helm}
(\Delta + k^2) \frac{1}{n-1}(\Delta u + k^2 n u) = 0 \, \text{ in } D
\end{align}
in the distributional sense. 
\\

{Notice that since both $v$ and $\Delta v$ are in $L^2(D)$ we have that the trace of $v$ on the boundary is in $H^{-1/2}(\pa D)$. }
We write the boundary condition (\ref{eig_bc1}) as
\begin{align}
\label{eig_diff_bc} \frac{1}{\eta}\frac{\pa u}{\pa \nu} = - \frac{1}{k^2(n-1)}(\Delta + k^2 n)u \,\, \,  \text{ on } \pa D.
\end{align}
Since {$\frac{1}{\eta}\frac{\pa u}{\pa \nu} \in L^2(\pa D) \subset  H^{-1/2}(\pa D)$}, the equality (\ref{eig_diff_bc}) is 
understood in {the ${H}^{-1/2}(\pa D)$} sense.
Combining (\ref{eig_diff_helm}) and (\ref{eig_diff_bc}) we arrive at a
variational formulation of (\ref{eig_helm_1})--(\ref{eig_bc2}) {by appealing to Green's second theorem}, which reads as follows: find {a nontrivial }
$u \in \tilde{H}^2_0(D)$ such that 
\begin{align}
&\Biggl\langle -\frac{1}{k^2(n-1)}(\Delta u + k^2 n u), \frac{\pa \varphi}{\pa \nu} \Biggr\rangle
= \Biggl\langle \frac{1}{\eta}\frac{\pa u}{\pa \nu}, \frac{\pa \varphi}{\pa \nu} \Biggr\rangle \nonumber \\
\label{eig_equiv_form1}
&\qquad \qquad \qquad  = \int\limits_D -\frac{1}{k^2(n-1)}(\Delta u + k^2 n u){\Delta \ov{\varphi} } - \frac{1}{n-1}(\Delta u + k^2 n u)\ov{\varphi} \, \dif x
\end{align}
for all $\varphi \in \tilde{H}^2_0(D)$ {where $\langle \cdot , \cdot \rangle$ is the dual pairing between ${H}^{-1/2}(\pa D)$ and ${H}^{1/2}(\pa D)$}.  Taking into account the regularity of $u$ and $\varphi$, and
multiplying both sides by $k^2$ the identity (\ref{eig_equiv_form1}) becomes:
\begin{align}\label{eig_var_form}
\int\limits_{\pa D} \frac{k^2}{\eta} \frac{\pa u}{\pa \nu} {\frac{\pa \ov{\varphi} }{\pa \nu}} \, \dif s
+ \int\limits_D \frac{1}{n-1}(\Delta u + k^2 n u)({\Delta \ov{\varphi} } + k^2 \ov{\varphi}) \, \dif x = 0, \quad \forall \; \varphi\in \tilde{H}^2_0(D)
\end{align}
The functions $v$ and $w$ are related to $u$ through
\[
v = -\frac{1}{k^2(n-1)}(\Delta u + k^2 n u) 
\quad \text{ and } \quad 
w = -\frac{1}{k^2(n-1)}(\Delta u + k^2 u). 
\]
\begin{definition} Values of $k>0$ for which the interior transmission problem
(\ref{eig_helm_1})--(\ref{eig_bc2}) has a nontrivial 
solution $v\in L^2(D)$ and $w\in L^2(D)$ such that 
$w-v\in \tilde{H}^2_0(D)$ are called transmission eigenvalues. If $k>0$ is a transmission eigenvalue, we call the 
solution $u \in \tilde{H}^2_0(D)$ of (\ref{eig_var_form}) the corresponding eigenfunction. 
\end{definition}
\section{Discreteness of the transmission eigenvalues}
In this section, we will prove that the set of transmission eigenvalues is at most discrete. To this end, we will write the transmission eigenvalue problem as a quadratic eigenvalues problem for $k^2$. 
Notice that from the variational formulation of the transmission eigenvalue problem (\ref{eig_var_form}) we have that the eigenvalue problem can now be written as
\begin{align}\label{eig_T}
\mathbb{T}u +k^2 \mathbb{T}_1 u +k^4 \mathbb{T}_2 u=0,
\end{align}
where the operator $ \mathbb{T}: \tilde{H}^2_0(D) \mapsto \tilde{H}^2_0(D)$ is  the bounded, self-adjoint operator defined by means of the Riesz representation theorem such that
\begin{align}\label{eig_op_T}
(\mathbb{T} u, \varphi )_{H^2(D)}=\int\limits_D \frac{1}{n-1} \Delta u \,  \Delta \overline{\varphi}  \, \dif x \quad \text{ for all } \quad  \varphi \in \tilde{H}^2_0(D). 
\end{align}
By Theorem 8.13 in \cite{Gilbarg} (note $\pa D \in C^2$), there 
exists a constant $C>0$ such that
\[
\|u\|^2_{H^2(D)} \leq C \Big(\|u\|^2_{L^2(D)} + \|\Delta u\|^2_{L^2(D)} \Big) \, \,  \text{ for all } u \in \tilde{H}^2_0(D).
\]
Since the trace of $u$ is zero we have that $\|u\|_{L^2(D)} \leq c \|\Delta u\|^2_{L^2(D)}$. Thus, the operator 
$\mathbb{T}$, for $n-1>0$, (or $-\mathbb{T}$, for $0<n<1$) is coercive on {$\tilde{H}^{2}_0(D)$} and, by the Lax-Milgram Lemma \cite{p1}, has a bounded inverse. Next, we define the operator $\mathbb{T}_1: \tilde{H}^2_0(D) \mapsto \tilde{H}^2_0(D)$ by means of the Riesz representation theorem such that $\text{for all }  \varphi \in \tilde{H}^2_0(D)$
\[
(\mathbb{T}_1 u, \varphi )_{H^2(D)}= - \int\limits_D \frac{1}{n-1} ( \overline{\varphi}  \, \Delta u + u \, \Delta \overline{\varphi}) \, \dif x +  \int\limits_D \grad u \cdot \grad \overline{\varphi} \, \dif x +  \int\limits_{\partial D} \frac{1}{\eta} \frac{\partial u}{\partial \nu} \frac{\partial \overline{\varphi} }{\partial \nu} \, \dif s. 
\]
The operator $\mathbb{T}_1 $ is self-adjoint and also compact.  
Indeed, if we define the auxiliary operator $\mathbb{A} :\tilde{H}^2_0(D) \mapsto \tilde{H}^2_0(D)$ such that 
\[
(\mathbb{A} u, \varphi )_{H^2(D)}= \int\limits_D \frac{1}{n-1} u \, \Delta \overline{\varphi} \, \dif x \quad \text{and} \quad (\mathbb{A}^* u, \varphi )_{H^2(D)}= \int\limits_D \frac{1}{n-1}  \overline{\varphi}  \, \Delta u\, \dif x.
\]
It is easy to see that $\| \mathbb{A}u \|_{H^2(D)}$ is bounded by $\| u \|_{L^2(D)}$. By Rellich's embedding theorem,
this implies that $\mathbb{A}$, and therefore $\mathbb{A}^*$, are compact. The compactness of $\mathbb{T}_1$ follows from the compactness of $\mathbb{A}$ and $\mathbb{A}^*$ along with the fact that $H^{1/2}(\partial D)$ is compactly embedded in $L^2(\partial D)$. 
At last, we define $\mathbb{T}_2:\tilde{H}^2_0(D) \mapsto \tilde{H}^2_0(D)$ by means of the Riesz representation theorem such that 
\[
(\mathbb{T}_2 u, \varphi )_{H^2(D)}=  \int\limits_D \frac{n}{ n-1 } u \, \overline{\varphi} \, \dif x  \quad \text{ for all } \quad  \varphi \in \tilde{H}^2_0(D). 
\]
$\mathbb{T}_2$ is compact and self-adjoint.

We are now ready to prove the discreteness of the set of transmission eigenvalues. 

\begin{theorem}
Assume that $n>1$ or  $0<n<1$  a.e. in $D$ and $\eta >0$ on $\partial D$ then the set of transmission eigenvalues is at most discrete. Moreover, the only accumulation point for the set of transmission eigenvalues is $+ \infty$. 
\end{theorem}
\begin{proof}
Let $\sigma=1$ when $n-1\geq \alpha>0$ and $\sigma=-1$ when $1-n \geq \alpha>0$. 
We write (\ref{eig_T}) as 
\[
 u +\sigma k^2 (\sigma \mathbb{T})^{-1} \mathbb{T}_1 u + \sigma k^4 (\sigma\mathbb{T})^{-1} \mathbb{T}_2 u=0
\]
or, equivalently (since $\sigma \mathbb{T}_2$ is a positive self-adjoint operator), as
\begin{align}\label{eig_matrix}
\left(\mathbb{K} -\frac{1}{k^2}\mathbb{I} \right)U =0
\end{align}
with $U = \big(u, k^2(\sigma \mathbb{T}_2)^{1/2}u\big)^\top \in \tilde{H}^2_0(D) \times \tilde{H}^2_0(D)$ and 
\[
\mathbb{K} = \begin{pmatrix} \sigma(\sigma\mathbb{T})^{-1}\mathbb{T}_1  \, \, & &(\sigma\mathbb{T})^{-1} \left(\sigma \mathbb{T}_2 \right)^{1/2}  \\ - \left(\sigma \mathbb{T}_2 \right)^{1/2} \, \,  & & 0 \end{pmatrix} : \tilde{H}^2_0(D) \times \tilde{H}^2_0(D) \mapsto \tilde{H}^2_0(D)\times \tilde{H}^2_0(D).
\]
The square root $\left(\sigma \mathbb{T}_2 \right)^{1/2}$ of the compact self-adjoint operator $\sigma \mathbb{T}_2$  is defined by 
$\displaystyle{ \left(\sigma \mathbb{T}_2 \right)^{1/2}= \int_0^\infty \lambda^{1/2} \dif E_\lambda}$, where $E_\lambda$ is the spectral measure associated 
with $\sigma \mathbb{T}_2$. The operator $\left(\sigma \mathbb{T}_2 \right)^{1/2}$ is compact and self-adjoint. 

We conclude that $\mathbb{K}$ is compact. From (\ref{eig_matrix}) we see that 
the interior eigenvalues $k$ are the inverse of the eigenvalues for the compact-matrix operator 
$\mathbb{K}$. Therefore, the interior eigenvalues form at most a discrete set with $\infty$ as the only accumulation point. Moreover, the eigenspaces for each eigenvalue have finite multiplicity.\end{proof}
\section{Existence of the transmission {eigenvalues} }
We prove the existence of infinitely many transmission eigenvalues using Theorem 2.3 of \cite{chtevexist}.
We recall this key result in the following lemma. 
\begin{lemma}\label{eig_key_lemma}(Theorem 2.3 of \cite{chtevexist})
Let $k \mapsto A_k$ be a continuous mapping from $(0,\infty)$ to the set of self-adjoint  positive definite bounded linear operators on the Hilbert space $U$ and assume that $B$ is a self-adjoint non-negative compact linear operator on $U$. We assume that there exist two positive constants $k_0$ and $k_1$ such that
\begin{enumerate}
\item $A_{k_0}-k_0^2B$ is positive on $U$
\item $A_{k_1}-k_1^2B$ is non-positive on a m dimensional subspace of $U$
\end{enumerate}
then each of the equations $\lambda_j(k)-k^2=0$ for $j=1, \dots, m$ has at least one solution in $[k_0,k_1]$ where $\lambda_j(k)$ is such that $A_{k}-\lambda_j(k) B$ has a non-trivial kernel. 
\end{lemma}

Recall the variational formulation of the transmission eigenvalue problem (\ref{eig_var_form}):
\begin{align}
\int\limits_{\pa D} \frac{k^2}{\eta} \frac{\pa u}{\pa \nu} {\frac{\pa\ov{ \varphi} }{\pa \nu}} \, \dif s
+ \int\limits_D \frac{1}{n-1}(\Delta u + k^2 n u)({\Delta \ov{\varphi} } + k^2 \ov{\varphi}) \, \dif x = 0, \quad \forall \; \varphi\in \tilde{H}^2_0(D).
\end{align}
We define the following bounded sesquilinear forms on $\tilde{H}^2_0(D)$:
\begin{align}
\hspace*{-1cm} \mathcal{A}_k(u,\varphi) &=  \int\limits_D  \frac{1}{ n-1} ( \Delta u +k^2 u )( \Delta \overline{\varphi} +k^2 \overline{\varphi} ) +k^4 u \overline{\varphi}  \, \dif x + k^2 \int\limits_{\partial D} \frac{1}{\eta} \frac{\partial u}{\partial \nu} \frac{\partial \overline{\varphi}}{\partial \nu} \, \dif s, \label{bad1} \\
\widetilde{\mathcal{A}}_k (u,\varphi) &= \int\limits_D  \frac{n}{ 1-n}( \Delta u +k^2 u ) ( \Delta \overline{\varphi} +k^2 \overline{\varphi} )+\Delta u \Delta \overline{\varphi}  \, \dif x, \label{bad2} \\
\mathcal{B}(u,\varphi) &= \int\limits_D \grad u \cdot \grad \overline{\varphi}  \, \dif x,  \quad \text{and} \\
\widetilde{\mathcal{B}}(u,\varphi) &= \int\limits_D \grad u \cdot \grad \overline{\varphi}  \, \dif x +\int\limits_{\partial D} \frac{1}{\eta} \frac{\partial u}{\partial \nu} \frac{\partial \overline{\varphi}}{\partial \nu} \, \dif s.   \label{bad4}
\end{align}
Now, we write the transmission eigenvalue problem either as
\begin{equation}
\mathcal{A}_k(u,\varphi) -k^2\mathcal{B}(u,\varphi) =0 \quad \text{ for all } \quad  \varphi \in \tilde{H}^2_0(D), \quad \text{where} \, \, \, n>1,
\end{equation}
or as
\begin{equation}
\widetilde{ \mathcal{A}}_k(u,\varphi) -k^2 \widetilde{ \mathcal{B}} (u,\varphi) =0 \quad \text{ for all } \quad  \varphi \in \tilde{H}^2_0(D), \quad \text{where} \, \, \, 0<n<1.
\end{equation}
Using the Riesz representation theorem we can define the bounded linear operators $\mathbb{A}_k$, $\widetilde{\mathbb{A}}_k$, $\mathbb{B}$, and $\widetilde{ \mathbb{B}}: \tilde{H}^2_0(D) \mapsto \tilde{H}^2_0(D)$ such that 
\begin{align*}
\left( {\mathbb{A}}_k u, \varphi \right)_{H^2(D)}={ \mathcal{A}}_k(u,\varphi),  \quad \left( \widetilde{\mathbb{A}}_k u, \varphi \right)_{H^2(D)}=\widetilde{ \mathcal{A}}_k(u,\varphi), 
\end{align*}
\begin{align*}
\left( {\mathbb{B}} u, \varphi \right)_{H^2(D)}={ \mathcal{B}}(u,\varphi)  \quad \text{ and } \quad \left( \widetilde{\mathbb{B}}u, \varphi \right)_{H^2(D)}=\widetilde{ \mathcal{B}}(u,\varphi).
\end{align*}
Since $n$ and $\eta$ are real valued the sesquilinear forms are Hermitian and therefore the operators are self-adjoint. 
Due to compact embeddings of $H^2(D)$ into $H^1(D)$ and $H^{1/2}(\partial D)$ into $L^2(\partial D)$ the 
operators $\mathbb{B}$ and $\tilde{ \mathbb{B}}$ are compact. Also since $\eta > 0$ both operators $\mathbb{B}$ and $\tilde{ \mathbb{B}}$ are positive (note that the trace of $u$ on $\partial D$ is zero). 

For the case when $n>1$ it has been shown in \cite{chtevexist} that   
 \begin{align*}
{ \mathcal{A}}_k(u,u) \geq C ||\Delta u||^2_{L^2(D)}  + k^2  \int\limits_{\partial D} \frac{1}{\eta} \left|\frac{\partial u}{\partial \nu} \right|^2 \, \dif s \geq C ||\Delta u||^2_{L^2(D)}
\end{align*}
where $C>0$ only depends on the refractive index $n$. Also for $\widetilde{\mathcal{A}}_k$, for the case $0<n<1$, we have
%
\begin{align*}
\widetilde{\mathcal{A}}_k(u,u) = \int\limits_D  \frac{n}{ 1-n}| \Delta u +k^2 u |^2 +|\Delta u|^2  \, \dif x \geq  ||\Delta u||^2_{L^2(D)}. 
\end{align*}
Therefore, for both $\mathcal{A}_k$ and $\widetilde{\mathcal{A}}_k$ holds
\[
\mathcal{A}_k(u,u) \geq C\|u\|_{H^2(D)} \quad \text{ and } \quad \widetilde{\mathcal{A}}_k(u,u) \geq c\|u\|_{H^2(D)} 
\]
for all $k\geq 0$, where the constants $C$ and $c$ are positive and independent of $u \in \tilde{H}^2_0(D)$. In the next theorem we summarize the properties of the 
operators $\mathbb{A}_k$, $\widetilde{\mathbb{A}}_k$, $\mathbb{B}$, and $\widetilde{ \mathbb{B}}$.
\begin{theorem}\label{tefredholm}
Assume that either $n>1$ or  $0<n<1$  a.e. in $D$ and that $\eta > 0$ on $\partial D$ then 
\begin{enumerate}
\item the operators $\mathbb{B} $ and $\widetilde{\mathbb{B}}$ are positive, compact, and self-adjoint. 
\item the operator $\mathbb{A}_k$ is a coercive self-adjoint operator provided that $n>1$. 
\item the operator $\widetilde{\mathbb{A}}_k $ is a coercive self-adjoint operator provided that $0<n<1$. 
\end{enumerate}
Therefore, the operators $\mathbb{A}_k -k^2 \mathbb{B}$ and $\widetilde{\mathbb{A}}_k -k^2 \widetilde{\mathbb{B}}$ satisfy the Fredholm property.  
\end{theorem}

Notice that the transmission eigenvalues are the solutions to $\lambda_j (k)-k^2=0$ where $\lambda_j(k)=\lambda_j(k ; n,\eta)$ are the eigenvalues for the generalized eigenvalue problem 
\begin{align}
\mathbb{A}_k u= \lambda_j(k) \mathbb{B} u \,\, \text{ for } \, \, 1<n \quad \text{ or } \quad \widetilde{\mathbb{A}}_k u= \lambda_j(k) \widetilde{\mathbb{B}} u  \,\, \text{ for } \, \, 0<n<1. \label{geneig} 
\end{align}
From the above discussion we have that $\mathbb{A}_k$, $\widetilde{\mathbb{A}}_k$, $\mathbb{B}_k$, and $\widetilde{ \mathbb{B}}_k$ satisfy the assumptions of Theorem 2.3 of \cite{chtevexist}. To prove existence it remains to show that the operators ${\mathbb{A}}_k-k^2{\mathbb{B}}$ and $\widetilde{\mathbb{A}}_k-k^2 \widetilde{\mathbb{B}}$ are positive for some $k_0$ and non-positive for some $k_1$ on a finite dimensional subspace of $\tilde{H}^2_0(D)$.

\begin{theorem}  \label{positive}
Assume that either $n>1$ or  $0<n<1$ a.e. in $D$ and $\eta > 0$ and $\partial D$ then for $k$ sufficiently small $\text{for all }  u \in \tilde{H}^2_0(D)$
\[
\mathcal{A}_k(u,u) -k^2\mathcal{B}(u,u) \geq \delta \| \Delta u \|^2_{L^2(D)}   \quad \text{ or  } \quad \widetilde{ \mathcal{A}}_k(u,u) -k^2 \widetilde{ \mathcal{B}} (u,u)  \geq \delta \| \Delta u \|^2_{L^2(D)}.
\]
\end{theorem}
\begin{proof}
We first consider the case where $0<n<1$ and since $\eta >0$ we have that 
\begin{align*}
\hspace*{-1cm} \widetilde{\mathcal{A}}_k(u,u) -k^2 \widetilde{\mathcal{B}}(u,u) &\geq  ||\Delta u||^2_{L^2(D)} -k^2 \left( ||\grad u||^2_{L^2(D)} +  \int\limits_{\partial D} \frac{1}{ \eta } \left| \frac{\partial u}{\partial \nu} \right|^2 \,  \dif s \right) \\
							&\geq   ||\Delta u||^2_{L^2(D)} -k^2 \left( || u ||^2_{H^2(D)} + \int\limits_{\partial D} \frac{1}{ \eta } \left| \frac{\partial u}{\partial \nu} \right|^2 \,   \dif s \right).
\end{align*}
Recall that the for all $u \in \tilde{H}_0^2(D)$ we have that there exists $C_1>0$ such that 
\[
|| u||^2_{H^2(D)} \leq C_1 ||\Delta u||^2_{L^2(D)}. 
\]
Now let $\inf\limits_{x \in \pa D} \eta =\eta_{min} >0$, then we have that $\frac{1}{\eta} \leq \frac{1}{\eta_{min}}$ for 
almost all $x \in \pa D$. Using these estimates yields that 
\begin{align*}
\hspace*{-0.5cm} \widetilde{\mathcal{A}}_k(u,u) -k^2\widetilde{\mathcal{B}}(u,u) \geq  ||\Delta u||^2_{L^2(D)} -k^2 \left( C_1 || \Delta u||^2_{L^2(D)} + \frac{1}{\eta_{min}} \left\| {\partial u}/{\partial \nu}\right\|^2_{L^2(\partial D)} \right). 
\end{align*}
By the trace theorem we obtain   
\[
\left\|\frac{\partial u}{\partial \nu}\right\|^2_{L^2(\partial D)} \leq C_2 || u||^2_{H^2(D)}. 
\]
Combining this with the previous estimates we conclude  
\begin{align*}
\widetilde{\mathcal{A}}_k(u,u) -k^2\widetilde{\mathcal{B}}(u,u) \geq \left[ 1-C_1k^2 \left(1+\frac{C_2}{\eta_{min}} \right) \right] ||\Delta u||^2_{L^2(D)}.
\end{align*}
Since $\eta_{min} >0$ we have that $\widetilde{\mathcal{A}}_k(u,u) -k^2\widetilde{\mathcal{B}}(u,u) \geq \delta \| \Delta u \|^2_{L^2(D)}$ for all $k>0$ sufficiently small.  

For $n>1$, since $\eta>0$ a.e. on $\pa D$, we have
\begin{align*}
&{\mathcal{A}}_k(u,u) -k^2 {\mathcal{B}}(u,u) \\
&\hspace{0.4in}= \int\limits_D  \frac{1}{ n-1} |\Delta u +k^2 u |^2 +k^4 |u|^2 \, \dif x  - k^2  \int\limits_D |\grad u|^2 \, \dif x+k^2 \int\limits_{\partial D} \frac{1}{\eta} \left| \frac{\partial u}{\partial \nu} \right|^2 \, \dif s \\
               &\hspace{0.4in}\geq C\| \Delta u\|^2_{L^2(D)} -k^2 \| \grad u\|^2_{L^2(D)}  \\
               &\hspace{0.4in}\geq  C\| \Delta u\|^2_{L^2(D)} - k^2 \|  u\|^2_{H^2(D)} \\
               &\hspace{0.4in}\geq (C-k^2 C_1) \| \Delta u\|^2_{L^2(D)},
\end{align*}
where again $C_1$ is the constant such that $||u||^2_{H^2(D)} \leq C_1 ||\Delta u||^2_{L^2(D)} $  for all $u \in \tilde{H}_0^2(D)$ and $C$ is the constant where 
\[ 
\int\limits_D  \frac{1}{ n-1} |\Delta u +k^2 u |^2 +k^4 |u|^2 \, \dif x  \geq C ||\Delta u||^2_{L^2(D)} \quad \text{for all}\quad u \in \tilde{H}_0^2(D).
\]
Hence, for all $k^2$ sufficiently small we have that ${\mathcal{A}}_k(u,u) -k^2 {\mathcal{B}}(u,u) \geq \delta \| \Delta u \|^2_{L^2(D)}$, proving the claim. 
\end{proof}

We are now ready to prove the main result of the paper.
\begin{theorem}  \label{exists}
Assume that either $n>1$ or  $0<n<1$  a.e. in $D$, then there exists infinitely many real transmission eigenvalues. 
\end{theorem}
\begin{proof}
We will prove the result for the case of $n>1$ and the other case is similar. Let $B_j=B(x_j , \ep):=\{ x \in \R^m : |x-x_j | <  \ep \}$ where $x_j \in D$ and $\ep>0$. Define $M(\ep)$ as the number of  disjoint balls $B_j$, i.e., $\overline{B_i} \cap \overline{B_j} = \emptyset$, with $\varepsilon$ small enough 
such that $\overline{B_j} \subset D$. 
%
%
It can be shown by using  separation of variables that there exists infinitely many transmission eigenvalues to 
\begin{align}
\Delta w_j +k^2 n_{min} w_j=0 \quad \text{and} \quad \Delta v_j + k^2 v_j=0  \quad &\textrm{ in } \,  B_j, \label{teball1} \\
 w_j-v_j=0  \quad  \text{and} \quad  \frac{\partial w_j}{\partial \nu}-\frac{\partial v_j}{\partial \nu}= 0 \quad &\textrm{ on } \partial B_j.  \label{teball2} 
\end{align}
where $n_{min}=\inf n(x)$ for $x \in D$. Let $u_j$ denote the difference $u_j=v_j-w_j \in H^2_0(B_j)$ and let $\tilde{u}_j$ be the extension of $u_j$ by zero to $D$. We note that $\tilde{u}_j \in H^2_0(D) \subset \tilde{H}_0^2(D)$. Since the supports of $\tilde{u}_j$ are disjoint we have that $\tilde{u}_j$ is orthogonal to $\tilde{u}_i$ for all $i \neq j$ in $\tilde{H}_0^2(D)$. This implies that $W_{M(\ep)} = \text{ span}\{ \tilde{u}_1 , \tilde{u}_2 , \cdots,  \tilde{u}_{M(\ep)}  \}$ forms an $M(\ep)$ dimensional subspace of $\tilde{H}_0^2(D)$. Further, for any transmission eigenvalue $k$ of \eqref{teball1}--\eqref{teball2} we have
\begin{align*}
0 &=  \int\limits_D \frac{1}{ n_{min}-1 }(\Delta \tilde{u}_j +k^2 \tilde{u}_j) (\Delta \overline{\tilde{u}_j } +k^2 n \overline{\tilde{u}_j}) \, \dif x\\
   &=  \int\limits_{B_j} \frac{1}{ n_{min}-1 }(\Delta \tilde{u}_j +k^2 \tilde{u}_j) (\Delta \overline{\tilde{u}_j } +k^2 n \overline{\tilde{u}_j}) \, \dif x \\
    &= \int\limits_{B_j}  \frac{1}{ n_{min}-1} | \Delta \tilde{u}_j +k^2 \tilde{u}_j |^2 +k^4 |\tilde{u}_j|^2 \, \dif x  - k^2 \int\limits_{ B_j} |\grad \tilde{u}_j|^2 \,  \dif x.
\end{align*}
Now, let $k_\ep$ be the first transmission eigenvalue of \eqref{teball1}--\eqref{teball2} in some ball $B_j$ with 
the eigenfunction $u_j$. Then, for the extension $\tilde{u}_j$  we have
\begin{align*}
\mathcal{A}_{k_{\ep}} (\tilde{u}_j, \tilde{u}_j ) -k_{\ep}^2\mathcal{B}(\tilde{u}_j,\tilde{u}_j) &\\
&\hspace{-1.2in}=\int\limits_{D}  \frac{1}{ n-1} | \Delta \tilde{u}_j +k_{\ep}^2 \tilde{u}_j |^2 +k_{\ep}^4 |\tilde{u}_j|^2 \, dx  - k_{\ep}^2 \int\limits_{D} |\grad \tilde{u}_j|^2 \, \dif x+k_{\ep}^2  \int\limits_{\partial D} \frac{1}{\eta} \left| \frac{\partial \tilde{u}_j}{\partial \nu} \right|^2 \,  \dif s\\
												 &\hspace{-1.2in}=  \int\limits_{D}  \frac{1}{ n-1} | \Delta \tilde{u}_j +k_{\ep}^2 \tilde{u}_j |^2 +k_{\ep}^4 |\tilde{u}_j|^2 \, \dif x  - k_{\ep}^2 \int\limits_{D} |\grad \tilde{u}_j|^2 \, \dif x\\
												                  &\hspace{-1.2in}\leq \int\limits_{B_j}  \frac{1}{ n_{min}-1} | \Delta \tilde{u}_j +k_{\ep}^2 \tilde{u}_j |^2 +k_{\ep}^4 |\tilde{u}_j|^2 \, \dif x  - k_{\ep}^2 \int\limits_{ B_j} |\grad \tilde{u}_j|^2 \, \dif x=0.
\end{align*}
Thus,  for all $u \in W_{M(\ep)}$, we have $\mathcal{A}_{k_{\ep}} ({u}, {u}) -k_{\ep}^2\mathcal{B}({u},{u}) \leq 0$. By Lemma \ref{eig_key_lemma} this gives that there are ${M(\ep)}$ transmission eigenvalues in the interval $(0 , k_\ep]$. Now, notice that as $\ep \rightarrow 0$ that $M(\ep) \rightarrow \infty$ giving that there are infinitely many transmission eigenvalues. 
\end{proof}

From the proof of Theorem \ref{exists} we have the following upper bound on the first 
transmission eigenvalue of \eqref{eig_helm_1}--\eqref{eig_bc2}, which we denote by $k_1(n , \eta , D)$. 
\begin{corollary}\label{bounds1}
Let $\sup_{x\in D} n(x)=n_{max}$ and $\inf_{x\in D} n(x)=n_{min}$. Let $B_R$ be a ball of radius $R$ sufficiently small such that $B_R \subseteq D$. Then
\begin{enumerate}
\item if $n>1$ for almost every $x \in D$, then
$$ k_1(n , \eta , D) \leq k_1( n_{min} , B_R),$$
\noindent where $k_1( n_{min} , B_R)$ is the first transmission eigenvalue of \eqref{teball1}--\eqref{teball2} for the ball $B_R$.
\item if $0<n<1$ for almost every $x \in D$, then 
$$ k_1(n , \eta , D) \leq k_1( n_{max} , B_R),$$
where $k_1( n_{manx} , B_R)$ is the first transmission eigenvalue of \eqref{teball1}--\eqref{teball2} for the ball $B_R$ with $n_{min}$ replaced by $n_{max}$. 
\end{enumerate}
\end{corollary}

The bound in Corollary \ref{bounds1} becomes sharp if $B_R$ is taken to be the largest ball such that $B_R \subseteq D$. 
%
 
\section{Monotonicity of the transmission eigenvalues}
For this section we turn our attention to proving that the first transmission eigenvalue can be used to determine information about the material parameters $n$ and $\eta$. To this end, we will show that the first transmission eigenvalue is a monotonic function with respect to the functions $n$ and $\eta$. From the monotonicity we will obtain a uniqueness result for a homogeneous refractive index and homogeneous conductive boundary parameter. Recall that the transmission eigenvalues satisfy 
\begin{equation}
\lambda_j (k; n ,\eta)-k^2( n ,\eta)=0 \label{teveq}
\end{equation}
and the first transmission eigenvalue is the smallest root of \eqref{teveq} for $\lambda_1 (k; n ,\eta)$. Notice that $\lambda_1 (k; n ,\eta)$ satisfies for $u \neq 0$
\begin{equation}
\lambda_1(k; n ,\eta) = \min_{u \in \tilde{H}^2_0(D)} \frac{ \mathcal{A}_k(u,u) }{  \mathcal{B}(u,u)}  \,\, \text{ for } \, \, 1<n \label{minprob1}
\end{equation}
or  
\begin{equation}
 \lambda_1(k; n ,\eta) =\min_{u \in \tilde{H}^2_0(D)} \frac{ \widetilde{\mathcal{A}}_k(u,u) }{ \widetilde{ \mathcal{B}}(u,u)}  \,\, \text{ for } \, \,0< n<1, \label{minprob2}
 \end{equation}
 where the sesquilinear forms on $\tilde{H}^2_0(D)$ are defined by \eqref{bad1}--\eqref{bad4}. It is clear that $\lambda_1(k; n ,\eta)$ is a continuous function of $k \in (0, \infty)$. Notice that the minimizers of \eqref{minprob1} and \eqref{minprob2} are the eigenfunctions corresponding to $ \lambda_1(k; n ,\eta)$. We will denote the first transmission eigenvalue as $k_1(n,\eta)$. 

\begin{theorem}\label{mono}
Assume that $0<n_1 \leq n_2$ and $0<\eta_1 \leq \eta_2 $, then we have that
\begin{enumerate}
\item if $1<n_1$, then $k_1(n_2 , \eta_2) \leq k_1(n_1 , \eta_1)$.
\item if $n_2 <1$, then $k_1(n_1 , \eta_1) \leq k_1(n_2 , \eta_2)$.
\end{enumerate}
Moreover, if the inequalities for the parameters $n$ and $\eta$ are strict, then the first transmission eigenvalue is strictly monotone with respect to $n$ and $\eta$. 
\end{theorem}
\begin{proof} 
We start with the case when $n>1$, where we let $k_1=k_1(n_1 , \eta_1)$ and $k_2=k_1(n_2 , \eta_2)$. Therefore, for all $u \in  \tilde{H}^2_0(D)$ such that $|| \grad u||_{L^2(D)}=1$ the inequalities $n_1\leq n_2$ and $\eta_1 \leq \eta_2 $ gives 
\begin{align*}
\hspace*{-1cm}\lambda_1(k_1 ; n_2 , \eta_2) &\leq  \int\limits_D  \frac{1}{ n_2 - 1 } |\Delta u +k_1^2 u|^2 + k_1^4| u|^2 \, \dif x  + k_1^2 \int\limits_{\partial D} \frac{1}{\eta_2} \left| \frac{\partial u}{\partial \nu} \right|^2  \, \dif s\\
					&\leq   \int\limits_D  \frac{1}{ n_1- 1} |\Delta u +k_1^2 u|^2 + k_1^4| u|^2  \, \dif x  +k_1^2 \int\limits_{\partial D} \frac{1}{\eta_1} \left| \frac{\partial u}{\partial \nu} \right|^2  \, \dif s. 
\end{align*}
Now, let $u=u_1$ where $u_1$ is the normalized transmission eigenfunction such that $|| \grad u_1 ||_{L^2(D)}=1$ corresponding with the eigenvalue $k_1$. Notice that \eqref{minprob1} gives  
\begin{align*}
\hspace*{-1cm}\lambda_1(k_1; n_1 , \eta_1)	=  \int\limits_D  \frac{1}{ n_1 -1 } |\Delta u_1 +k_1^2 u_1|^2 + k_1^4 | u_1|^2  \, \dif x  + k_1^2 \int\limits_{\partial D} \frac{1}{\eta_1} \left| \frac{\partial u_1}{\partial \nu} \right|^2  \, \dif s,
\end{align*}
since $u_1$ is the minimizer of \eqref{minprob1} for $n=n_1$ and $\eta=\eta_1$. This yields $\lambda_1(k_1 ; n_2 , \eta_2) \leq \lambda_1(k_1; n_1 , \eta_1)$ and \eqref{teveq} gives 
$$\lambda_1(k_1 ; n_1 , \eta_1)-k_1^2=0.$$ 
Recall that $\lambda_1(k_1 ; n_2 , \eta_2) -k_1^2 \leq 0$. Now, for all $k^2$ sufficiently small we have that ${\mathcal{A}}_k(u,u) -k^2 {\mathcal{B}}(u,u) >0$ by Theorem \ref{positive}. This implies that there is a $\delta >0$ such that for any $k^2 < \delta$ that $\lambda_1(k ; n_2 , \eta_2) -k^2 > 0$ holds. By the continuity we have that $\lambda_1(k ; n_2 , \eta_2) -k^2$ has at least one root in the interval $\left[ \sqrt{ \delta },  k_2 \right]$, since $k_2$ is the smallest root of $\lambda_1(k ; n_2 , \eta_2) -k^2$ we conclude that $k_2 \leq k_1$ proving the claim for this case. 

For the case where $n_2 <1$ we let $k_1=k_1(n_1 , \eta_1)$ and $k_2=k_1(n_2 , \eta_2)$ and the corresponding sesquilinear forms 
\begin{align*}
\widetilde{\mathcal{A}}_k (u,\varphi) &= \int\limits_D  \frac{n}{ 1-n} ( \Delta u +k^2 u )( \Delta \overline{\varphi} +k^2 \overline{\varphi} ) +\Delta u \, \Delta \overline{\varphi}  \, \dif x, \\
\widetilde{\mathcal{B}}(u,\varphi) &= \int\limits_D \grad u \cdot \grad \overline{\varphi}  \, \dif x +  \int\limits_{\partial D} \frac{1}{\eta} \frac{\partial u}{\partial \nu} \frac{\partial \overline{\varphi}}{\partial \nu} \, \dif s.   
\end{align*}

Recall that 
$$  \lambda_1(k; n_1 ,\eta_1) =\min\limits_{u \in \tilde{H}^2_0(D)} \frac{ \widetilde{\mathcal{A}}_k(u,u) \big|_{n=n_1}}{ \widetilde{ \mathcal{B}}(u,u) \big|_{ \mu= \eta_1}},  $$
where we have assumed that $n_1\leq n_2$ and $\eta_1 \leq \eta_2 $ we have that for any value $k$ and for all $u \in  \tilde{H}^2_0(D)$
\begin{align*}
 \int\limits_D  \frac{n_1}{ 1-n_1} | \Delta u +k^2 u |^2 + |\Delta u|^2   \, \dif x  &\leq  \int\limits_D  \frac{n_2}{ 1-n_2} | \Delta u +k^2 u |^2 + |\Delta u|^2  \, \dif x , \\
\int\limits_D |\grad u |^2  \, \dif x +  \int\limits_{\partial D} \frac{1}{\eta_2} \left| \frac{\partial u}{\partial \nu} \right|^2  \, \dif s  &\leq \int\limits_D |\grad u |^2  \, \dif x +  \int\limits_{\partial D} \frac{1}{\eta_1} \left| \frac{\partial u}{\partial \nu} \right|^2  \, \dif s.   
\end{align*}
This gives $\widetilde{\mathcal{A}}_k (u,u) \big|_{ n= n_1} \leq \widetilde{\mathcal{A}}_k (u,u) \big|_{ n= n_1}$ and $\, \widetilde{\mathcal{B}}(u,u) \big|_{ \eta= \eta_2} \leq \widetilde{\mathcal{B}}(u,u) \big|_{ \eta= \eta_1}$.
Now by letting $u=u_2$ where $u_2$ is the transmission eigenfunction corresponding with transmission eigenvalue $k_2$ we have that 
$$ \lambda_1(k_2; n_1 ,\eta_1) \leq \frac{ \widetilde{\mathcal{A}}_{k_2} (u_2 ,u_2 ) \big|_{ n= n_1} }{ \widetilde{\mathcal{B}}(u_2,u_2)  \big|_{ \eta=\eta_1} } \leq \frac{ \widetilde{\mathcal{A}}_{k_2} (u_2,u_2) \big|_{ n= n_2} }{ \widetilde{\mathcal{B}}(u_2,u_2)  \big|_{ \eta= \eta_2} }.  $$
Using that $u_2$ is the minimizer for \eqref{minprob2}  for $n=n_2$ and $\eta=\eta_2$ we can conclude that $ \lambda_1(k_2 ; n_1 ,\eta_1)  \leq  \lambda_1(k_2 ; n_2 ,\eta_2)$ and similar arguments as in the previous case gives $k_1 \leq k_2$.
\end{proof}

By the proof of the previous result we have the following uniqueness result for a homogeneous media and homogeneous boundary parameter $\eta$ from the strict monotonicity of the first transmission eigenvalue. 
\begin{corollary}
\begin{enumerate} 
\item If it is known that $n>1$ or $0<n<1$ is a constant refractive index with $\eta$ known and fixed, then $n$ is uniquely determined by the first transmission eigenvalue. 
\item If $n>1$ or $0<n<1$ is known and fixed with $\eta$ a constant, then the first transmission eigenvalue uniquely determines $\eta$. \\
\end{enumerate}
\end{corollary}

It is known (see \cite{chtevexist}) that for a every fixed $k \in (0 , \infty)$ there exists an increasing sequence $\lambda_j(k ; n, \eta)$ of positive generalized eigenvalues of \eqref{geneig} that satisfy 
\begin{align*}
\lambda_j(k; n ,\eta) = \min\limits_{ U \in \mathcal{U}_j} \max\limits_{u \in U\setminus \{ 0\} }  \frac{ \mathcal{A}_k(u,u) }{  \mathcal{B}(u,u)}  \,\, \text{ for } \, \, 1<n 
\end{align*}
or
\begin{align*}
\lambda_j(k; n ,\eta) =\min\limits_{ U \in \mathcal{U}_j} \max\limits_{u \in U \setminus \{ 0\} }  \frac{ \widetilde{\mathcal{A}}_k(u,u) }{ \widetilde{ \mathcal{B}}(u,u)}  \,\, \text{ for } \, \,0< n<1,
\end{align*}
where $\mathcal{U}_j$ is the set of all $j$-dimensional subspaces of $\tilde{H}^2_0(D)$. It is clear from the proof of Theorem \ref{mono} that if $k_j$ is a transmission eigenvalue such that $\lambda_j(k ; n, \eta)-k^2=0$, then $k_j(n,\eta)$ satisfies the monotonicity properties given in Theorem \ref{mono}. 

\begin{corollary}
Assume that $0<n_1 \leq n_2$ and $0<\eta_1 \leq \eta_2 $ and that $k_j$ is a transmission eigenvalue such that $\lambda_j(k)-k^2=0$, where $\lambda_j(k)$ is a positive generalized eigenvalues of \eqref{geneig}, then we have:
\begin{enumerate}
\item if $1<n_1$, then we have that $k_j(n_2 , \eta_2) \leq k_j(n_1 , \eta_1)$.
\item if $n_2 <1$, then we have that $k_j(n_1 , \eta_1) \leq k_j(n_2 , \eta_2)$.
\end{enumerate}
\end{corollary}

\section{Numerical results}
In this section, we present some numerical results for three obstacles in three dimensions to validate the theoretical results of the previous sections. The obstacles under consideration are a unit sphere centered at the origin, a peanut-shaped object, and a cushion-shaped object. The obstacles are shown in Figure \ref{obstacles}. Their parametrization in spherical coordinates is described in \cite[Section 6]{kleefeldITP} and given later for the sake of completeness.
\begin{figure}[htb] 
\centering
		\subfigure{ 
		\includegraphics[height=3.7cm]{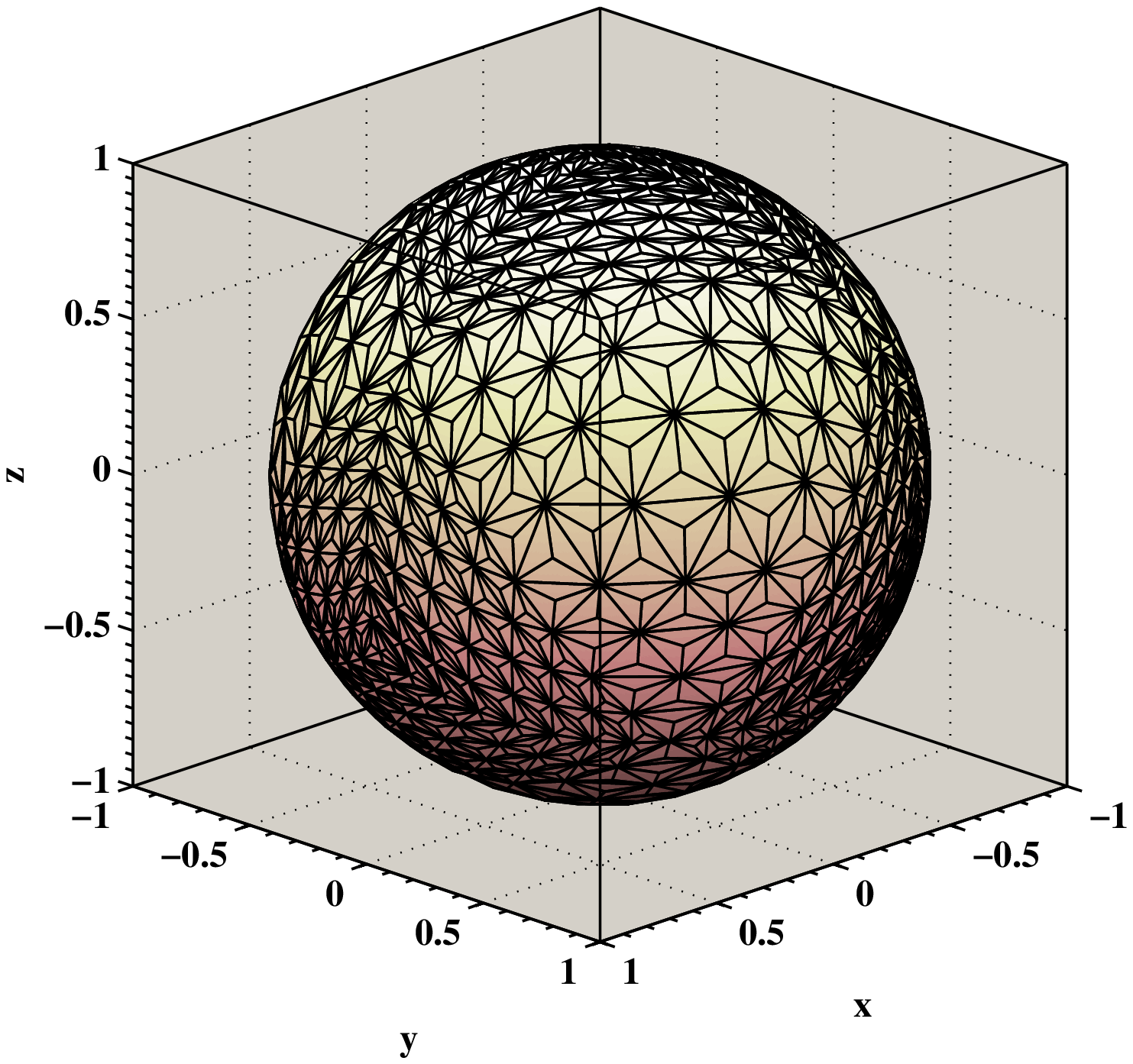}
}
\subfigure{ 
\includegraphics[height=3.7cm]{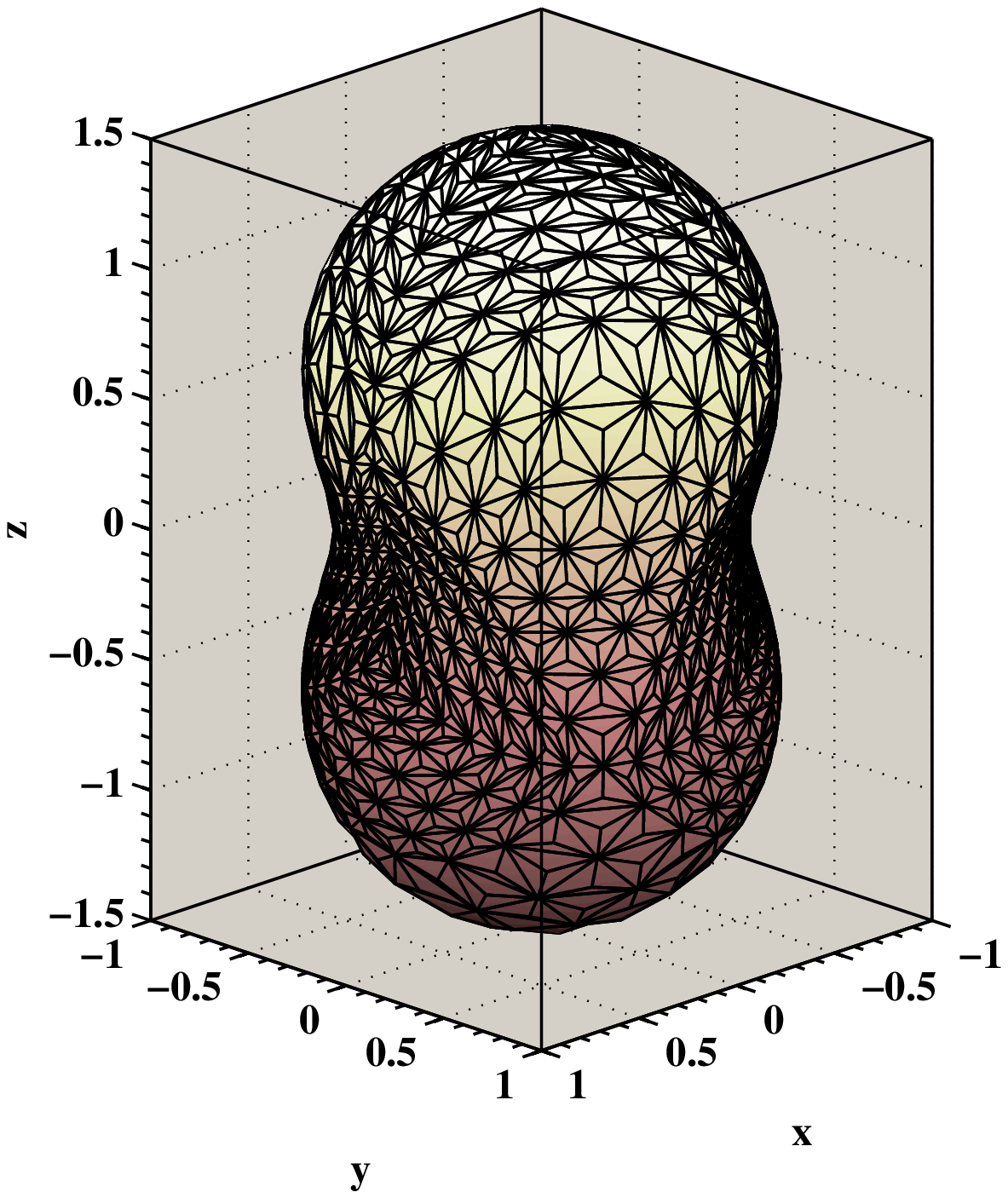}
}
\subfigure{ 
\includegraphics[height=3.7cm]{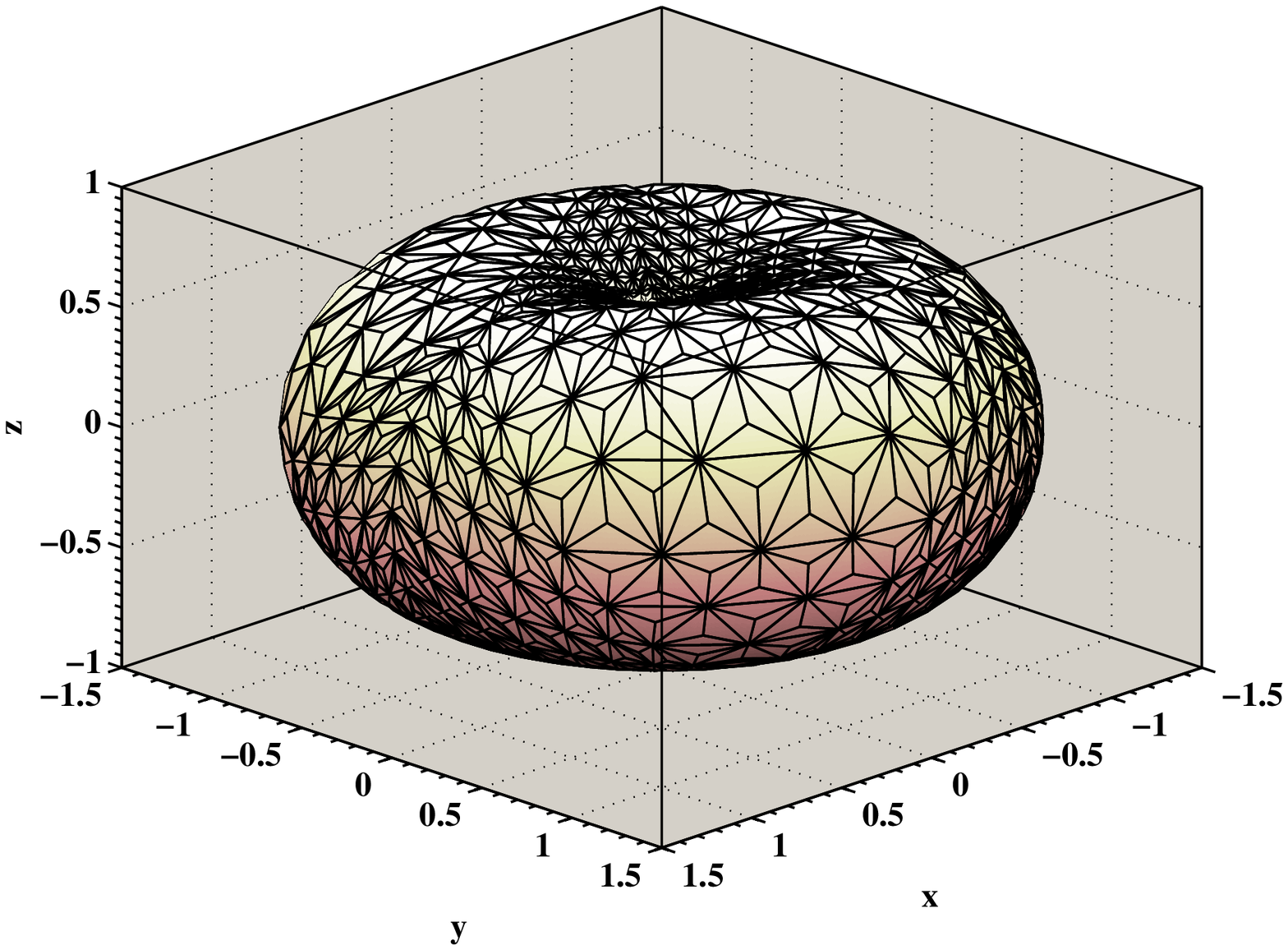}
}
		\caption{\label{obstacles}Left to right: Unit sphere centered at the origin, peanut-shaped obstacle, and cushion-shaped obstacle.}
		\end{figure}
		
First, we numerically calculate the interior transmission eigenvalues for a sphere of radius $R>0$ centered at the origin. 
This is achieved with a series expansion.
Then, 
we solve the problem at hand via a system of boundary integral equations and approximate it numerically with a boundary element collocation method. The resulting nonlinear eigenvalue problem is numerically solved with complex-valued contour integrals to obtain the interior transmission eigenvalues for a variety of surfaces for various parameter settings.
\subsection{Interior transmission eigenvalues for a sphere}
First, we calculate the interior transmission eigenvalues of a sphere with radius $R>0$ centered at the origin.
The solution to the interior transmission problem (\ref{teprob1})--(\ref{teprob2}) can be written as
\begin{eqnarray*}
v(r\hat{x})&=&\sum_{p=0}^\infty \sum_{m=-p}^{p} \alpha_p^m j_p(kr)Y_p^m(\hat{x})\,,\quad r<R\,,\\
w(r\hat{x})&=&\sum_{p=0}^\infty \sum_{m=-p}^{p} \beta_p^m j_p(k\sqrt{n}r)Y_p^m(\hat{x})\,,\quad r<R\,,
\end{eqnarray*}
where $x=r\hat{x}$ with $r>0$ and $\hat{x}\in \mathbb{S}=\{x\in\mathbb{R}^3:|x|=1\}$. Here, $j_p$ denotes the spherical Bessel function of the first kind of order $p$ and $Y_p^m$ is the spherical wave function.  
Using the boundary condition $w-v=0$ on the sphere of radius $R$ yields
\begin{equation}
\beta_p^m j_p(k\sqrt{n}R)-\alpha_p^m j_p(kR)=0\,.
\label{first}
\end{equation}
Likewise, using the boundary condition $\partial_r w-\partial_r v-\eta v=0$ on the sphere of radius $R$ gives
\begin{equation}
\beta_p^m k\sqrt{n}j_p'(k\sqrt{n}R)-\alpha_p^m k j_p'(kR)-\alpha_p^m \eta j_p(kR)=0\,.
\label{second}
\end{equation}
Equations (\ref{first}) and (\ref{second}) can be written as
\begin{eqnarray*}
\left(
\begin{array}{lr}
	-j_p(kR) \,\,\,& \quad j_p(k\sqrt{n}R)\\
	-k j_p'(kR)- \eta j_p(kR)\,\,\, &\quad k\sqrt{n}j_p'(k\sqrt{n}R)
\end{array}
\right)\left(
\begin{array}{cc}
	\alpha_p^m\\
	\beta_p^m
\end{array}
\right)=\left(\begin{array}{cc}
	0\\
	0
\end{array}\right).
\end{eqnarray*}
With the definition $M_p(k)$ for the matrix, we have to numerically calculate the zeros of
\begin{eqnarray}
\det M_p(k)
\label{inspect}
\end{eqnarray}
for $p\geq 0$
to find the interior transmission eigenvalues. The parameters $n$ and $\eta$ are given. In Table \ref{table1}, we list the first six interior transmission eigenvalues for a unit sphere using the index of refraction $n=4$ and various choices of $\eta$.
\begin{table}[!ht]
\centering
\begin{tabular}{l|ccccccccc}
     $\eta$     & 1. & 2. & 3. & 4. & 5. & 6. \\
 \hline
 $0.01$ & 3.136\,675 & 3.140\,531 & 3.141\,593 & 3.691\,542 & 4.260\,901 & 4.831\,165\\
 $0.1 $ & 3.109\,444 & 3.130\,912 & 3.141\,593 & 3.683\,405 & 4.253\,868 & 4.824\,974\\
 $0.25$ & 3.059\,806 & 3.114\,638 & 3.141\,593 & 3.669\,807 & 4.242\,177 & 4.814\,701\\
 $0.5$  & 2.974\,096 & 3.086\,914 & 3.141\,593 & 3.647\,091 & 4.222\,806 & 4.797\,750\\
 $1$    & 2.798\,386 & 3.029\,807 & 3.141\,593 & 3.601\,813 & 4.184\,685 & 4.764\,588\\
 $2$    & 2.458\,714 & 2.914\,716 & 3.141\,593 & 3.514\,484 & 4.112\,257 & 4.701\,954\\
 $3$    & 2.204\,525 & 2.809\,294 & 3.141\,593 & 3.435\,429 & 4.046\,733 & 4.645\,150\\
 $10$   & 1.743\,402 & 2.467\,800 & 3.138\,749 & 3.141\,593 & 3.779\,199 & 4.399\,490\\
 $100$  & 1.586\,662 & 2.269\,209 & 2.910\,355 & 3.141\,593 & 3.528\,384 & 3.904\,038\\
 $1000$ & 1.572\,369 & 2.248\,952 & 2.884\,610 & 3.141\,593 & 3.497\,455 & 3.866\,514\\
 $10000$& 1.570\,953 & 2.246\,929 & 2.882\,018 & 3.141\,593 & 3.494\,315 & 3.863\,012\\
 \hline
\end{tabular}
\caption{\label{table1}The first six interior transmission eigenvalues for a unit sphere using the index of refraction $n=4$ and various choices of $\eta$.}
\end{table}

As we can see in Table \ref{table1}, we obtain for the limiting case $\eta=0$ the `classic' interior transmission eigenvalues $3.141593$, $3.692445$, $4.261683$, and $4.831855$ (see for example \cite[Table 12]{kleefeldITP}). Interestingly, the first three interior transmission eigenvalues converge to $3.141593$ as $\eta\rightarrow 0$. One can also observe that there is a crossover of the third and fourth interior eigenvalue between $\eta=3$ and $\eta=10$. The limiting case for $\eta=\infty$ gives the union of the interior Dirichlet eigenvalues for a unit sphere and a sphere of radius two which can easily seen by considering the limiting case in (\ref{inspect}). The values are given by the zeros of $j_p(k)$ and $j_p(2k)$, respectively. The first four interior Dirichlet eigenvalues for a unit sphere are $3.141593$, $4.493408$, $5.236630$, and $5.763441$ (see also \cite[Table 11]{kleefeldITP}). The first five interior Dirichlet eigenvalues for a sphere of radius two are $1.570796$, $2.246705$, $2.881730$, $3.493966$, and $3.862626$.  
 
Next, we fix $\eta=0,1$ and show the change of the second interior transmission eigenvalue for various $n$. As we can see in Table \ref{table2}, we have a large shift of the second interior transmission eigenvalue to the right for both choices of $\eta$, if $n$ is decreases close to one.
\begin{table}[!ht]
\centering
\begin{tabular}{l|rr}
      $n$    & $\eta=0$ & $\eta=1$ \\
 \hline
 4   & 3.141\,593 & 3.029\,807\\
 3.5 & 3.457\,508 & 3.330\,436\\
 3   & 4.101\,812 & 3.933\,890\\
 2.5 & 5.744\,627 & 5.632\,859\\
 2.3 & 6.162\,456 & 6.036\,482\\
 2.1 & 6.734\,597 & 6.586\,329\\
 2.0 & 7.358\,550 & 7.164\,547\\
 1.9 & 8.745\,665 & 8.628\,245\\
 1.8 & 9.294\,075 & 9.160\,258\\
 1.7& 10.000\,772 & 9.841\,243\\
 \hline
\end{tabular}
\caption{\label{table2}The second interior transmission eigenvalues for a unit sphere using various index of refractions $n$ and $\eta=0$ and $\eta=1$.}
\end{table}
Lastly, we numerically show that the estimated order of convergence for a variety of interior transmission eigenvalues as $\eta$ goes to zero seems to be linear. Therefore, we define the absolute error $\epsilon_\eta^{(i)}=|k_0^{(i)}-k_\eta^{(i)}|$ for the $i$-th interior transmission eigenvalue. The estimated order of convergence is given by $\mathrm{EOC}^{(i)}=\log(\epsilon_\eta^{(i)}/\epsilon_{\eta/2}^{(i)})/\log(2)$. In Table \ref{table3} we list the absolute error and the estimated order of convergence for the second, fourth, and sixth interior transmission eigenvalue for a unit sphere using $n=4$.
\begin{table}[!ht]
\centering
\begin{tabular}{c|cc|cc|cc}
      $\eta$    & $\epsilon_\eta^{(2)}$ & $\mathrm{EOC}^{(2)}$ & $\epsilon_\eta^{(4)}$ & $\mathrm{EOC}^{(4)}$ & $\epsilon_\eta^{(6)}$ & $\mathrm{EOC}^{(6)}$\\
 \hline
 1    &$1.118\cdotp 10^{-1}$ &       &$9.063\cdotp 10^{-2}$ &      &$6.727\cdotp 10^{-2}$ &      \\
 1/2  &$5.468\cdotp 10^{-2}$ & 1.032 &$4.535\cdotp 10^{-2}$ & 0.999&$3.411\cdotp 10^{-2}$ & 0.980\\
 1/4  &$2.695\cdotp 10^{-2}$ & 1.021 &$2.264\cdotp 10^{-2}$ & 1.002&$1.716\cdotp 10^{-2}$ & 0.991\\
 1/8  &$1.337\cdotp 10^{-2}$ & 1.011 &$1.130\cdotp 10^{-2}$ & 1.003&$8.601\cdotp 10^{-3}$ & 0.996\\
 1/16 &$6.659\cdotp 10^{-3}$ & 1.006 &$5.647\cdotp 10^{-3}$ & 1.001&$4.306\cdotp 10^{-3}$ & 0.998\\
 1/32 &$3.323\cdotp 10^{-3}$ & 1.003 &$2.822\cdotp 10^{-3}$ & 1.001&$2.154\cdotp 10^{-3}$ & 0.999\\
 1/64 &$1.660\cdotp 10^{-3}$ & 1.001 &$1.411\cdotp 10^{-3}$ & 1.000&$1.078\cdotp 10^{-3}$ & 0.999\\
 1/128&$8.294\cdotp 10^{-4}$ & 1.001 &$7.050\cdotp 10^{-4}$ & 1.001&$5.388\cdotp 10^{-4}$ & 1.001\\
 1/256&$4.146\cdotp 10^{-4}$ & 1.000 &$3.523\cdotp 10^{-4}$ & 1.001&$2.694\cdotp 10^{-4}$ & 1.000\\
 \hline
\end{tabular}
\caption{\label{table3}The estimated order of convergence for the second, fourth, and sixth interior transmission eigenvalue for a unit sphere using $n=4$ as $\eta\rightarrow 0$.}
\end{table}

\subsection{Interior transmission eigenvalues for arbitrary obstacles}
First, we will derive the system of boundary integral equations to solve the interior transmission problem (\ref{teprob1})--(\ref{teprob2}) which is an easy extension of Cossonni\`{e}re and Haddar (see \cite{cossohaddar}). Later, we will approximate this system of boundary integral equations to numerically compute the interior transmission eigenvalues. 
\subsubsection{A system of boundary integral equations}
We apply Green's representation theorem in $D$ to obtain (see \cite[Theorem 2.1]{coltonkress})
\begin{eqnarray}
w(x)&=&SL_{k\sqrt{n}}(\partial_{\nu}w)(x)-DL_{k\sqrt{n}} (w)(x)\,,\qquad x\in D\,,\label{start2}\\
v(x)&=&SL_{k}(\partial_{\nu}v)(x)-DL_{k} (v)(x)\,,\qquad x\in D\,,\label{start1}
\end{eqnarray}  
where 
\begin{eqnarray*}
SL_k(f)(x)&=&\int_{\partial D}\Phi_k(x,y)f(y)\;\mathrm{d}s(y)\,,\qquad x\in D\,,\\
DL_k(f)(x)&=&\int_{\partial D}\partial_{\nu(y)}\Phi_k(x,y)f(y)\;\mathrm{d}s(y)\,,\qquad x\in D\,,
\end{eqnarray*}
are the single and double layer potentials that are defined for points in the domain $D$, respectively. 
Here, the function $\Phi_k(x,y)$ is the fundamental solution of the Helmholtz equation depending on the wave number $k$.

Now, we let the point $x\in D$ approach the boundary $\partial D$ and then we use the jump relations of the single and double layer potentials (see for example \cite[Theorem 3.1]{coltonkress}) to obtain
\begin{eqnarray}
w(x)&=&S_{k\sqrt{n}}(\partial_{\nu}w)(x)-D_{k\sqrt{n}} (w)(x)+\frac{1}{2}w(x)\,,\qquad x\in \partial D\,,\label{one}\\
v(x)&=&S_{k}(\partial_{\nu}v)(x)-D_{k} (v)(x)+\frac{1}{2}v(x)\,,\qquad x\in \partial D\,,\label{two}
\end{eqnarray}
where 
\begin{eqnarray*}
S_k(f)(x)&=&\int_{\partial D}\Phi_k(x,y)f(y)\;\mathrm{d}s(y)\,,\qquad x\in \partial D\,,\\
D_k(f)(x)&=&\int_{\partial D}\partial_{\nu(y)}\Phi_k(x,y)f(y)\;\mathrm{d}s(y)\,,\qquad x\in \partial D\,,
\end{eqnarray*}
are the single and double layer boundary integral operators, respectively. 

Next, we apply the boundary condition, then equation (\ref{one}) can be written as
\begin{eqnarray}
v(x)=S_{k\sqrt{n}}(\partial_{\nu}v)(x)+\eta S_{k\sqrt{n}}(v)(x) -D_{k\sqrt{n}} (v)(x)+\frac{1}{2}v(x)\,,\qquad x\in \partial D\,\label{twonew}.
\end{eqnarray}
Finally, we take the difference of (\ref{twonew}) and (\ref{two}) and obtain
\begin{eqnarray}
0=\left(S_{k\sqrt{n}}-S_{k}\right)(\partial_{\nu}v)(x)+\left(-D_{k\sqrt{n}}+D_{k}+\eta S_{k\sqrt{n}}\right)(v)(x)\,,\qquad x\in\partial D\,.\label{important1a}
\end{eqnarray}
Now, we take the normal derivative of equations (\ref{start2}) and (\ref{start1}), then we let the point $x\in D$ approach the boundary $\partial D$, and use the jump relations of the normal derivative of the single and double layer potentials (see for example \cite[Theorem 3.1]{coltonkress}) which gives
\begin{eqnarray}
\partial_{\nu} w(x)&=&K_{k\sqrt{n}}(\partial_{\nu}w)(x)+\frac{1}{2}\partial_{\nu}w(x)-T_{k\sqrt{n}} (w)(x)\,,\qquad x\in \partial D\,,\label{start2nj}\\
\partial_{\nu} v(x)&=&K_{k}(\partial_{\nu}v)(x)+\frac{1}{2}\partial_{\nu}v(x)-T_{k} (v)(x)\,,\qquad x\in \partial D\,,\label{start1nj}
\end{eqnarray}
where 
\begin{eqnarray*}
K_k(f)(x)&=&\int_{\partial D}\partial_{\nu(x)}\Phi_k(x,y)f(y)\;\mathrm{d}s(y)\,,\qquad x\in \partial D\,,\\
T_k(f)(x)&=&\partial_{\nu(x)}\int_{\partial D}\partial_{\nu(y)}\Phi_k(x,y)f(y)\;\mathrm{d}s(y)\,,\qquad x\in \partial D\,,
\end{eqnarray*}
are the normal derivative of the single and double layer boundary integral operators, respectively. 
Next, we apply the boundary conditions, then equation (\ref{start2nj}) can be written as
\begin{eqnarray}
&&\partial_{\nu} v(x)+\eta v(x)\label{start2njnew}\\
&&\hspace{0.2in}= K_{k\sqrt{n}}(\partial_{\nu}v)(x)+\eta K_{k\sqrt{n}}(v)(x) +\frac{1}{2}\partial_{\nu}v(x)+\frac{1}{2}\eta v(x)-T_{k\sqrt{n}} (v)(x)\,,\quad x\in \partial D\,.\notag
\end{eqnarray}
Finally, we take the difference of (\ref{start2njnew}) and (\ref{start1nj}) and obtain
\begin{eqnarray}
0=\left(K_{k\sqrt{n}}-K_{k}\right)(\partial_{\nu}v)(x)+\left(-T_{k\sqrt{n}}+T_{k}+\eta K_{k\sqrt{n}}-\frac{1}{2}\eta I\right)(v)(x)\,,\label{important2a}
\end{eqnarray}
where $x\in\partial D$.
Using the notation $\alpha=\partial_{\nu}v(x)$ and $\beta=v(x)$, $x\in\partial D$, we can write (\ref{important1a}) and (\ref{important2a}) as
\begin{eqnarray}
\left(
\begin{array}{cc}
S_{k\sqrt{n}}-S_{k} & \quad \quad -D_{k\sqrt{n}}+D_{k}+\eta S_{k\sqrt{n}}\\
K_{k\sqrt{n}}-K_{k} & \quad \quad -T_{k\sqrt{n}}+T_{k}+\eta \left(K_{k\sqrt{n}}-\frac{1}{2}I\right)
\end{array}
\right)
\left(
\begin{array}{c}
\alpha\\
\beta
\end{array}
\right)=
\left(
\begin{array}{c}
0\\
0
\end{array}
\right).
\label{equation}
\end{eqnarray}
The boundary integral operator $Z$ depending on the wave number $k$ is given by
\begin{eqnarray}
Z(k)=
\left(
\begin{array}{cc}
S_{k\sqrt{n}}-S_{k} & \quad \quad -D_{k\sqrt{n}}+D_{k}+\eta S_{k\sqrt{n}}\\
K_{k\sqrt{n}}-K_{k} & \quad \quad -T_{k\sqrt{n}}+T_{k}+\eta \left(K_{k\sqrt{n}}-\frac{1}{2}I\right)
\end{array}
\right)
\label{matrix}
\end{eqnarray}
and hence (\ref{equation}) can be written abstractly as
\begin{eqnarray}
Z(k)X=0
\label{equation2}
\end{eqnarray}
with the obvious definition of $X$. 

{Note that the boundary integral operator $Z(k):H^{-3/2}(\partial D)\times H^{1/2}(\partial D)\mapsto$ $H^{1/2}(\partial D)\times H^{-1/2}(\partial D)$ for $\eta>0$ is of Fredholm type with index zero, and it is analytic on the upper half-plane of $\mathbb{C}$. Indeed, notice that $Z(k)$ can be written as $Z(k) = C(k)+T(k)$ with 
\begin{eqnarray*}
T(k)=
\left(
\begin{array}{cc}
-\gamma(k,k\sqrt{n}) \left(S_{i |k\sqrt{n}|}-S_{i|k|} \right) & \quad \quad \eta S_{i |k\sqrt{n}|} \\
-\gamma(k,k\sqrt{n}) \left(K_{i |k\sqrt{n}|}-K_{i|k|} \right)& \quad \quad \eta \left(K_{i |k\sqrt{n}|}-\frac{1}{2}I\right)
\end{array}
\right)
\end{eqnarray*}
and 
\begin{eqnarray*}
&&C(k) \\
&=&\left(
\begin{array}{cc}
S_{k\sqrt{n}}-S_{k}+\gamma(k,k\sqrt{n}) \left(S_{i |k\sqrt{n}|}-S_{i|k|} \right) & \quad  -D_{k\sqrt{n}}+D_{k}+\eta \left( S_{k\sqrt{n}}- S_{i |k\sqrt{n}|} \right) \\
K_{k\sqrt{n}}-K_{k}-\gamma(k,k\sqrt{n}) \left(K_{i |k\sqrt{n}|}-K_{i|k|} \right)& \quad  -T_{k\sqrt{n}}+T_{k}+\eta \left( K_{k\sqrt{n}} - K_{i |k\sqrt{n}|} \right)
\end{array}
\right)
\end{eqnarray*}
where the constant ${ \gamma(a,b)=\frac{a^2-b^2}{|a|^2-|b|^2}}$. Therefore, the compactness of the operator $C(k):H^{-3/2}(\partial D)\times H^{1/2}(\partial D)\mapsto$ $H^{1/2}(\partial D)\times H^{-1/2}(\partial D)$ is given by \cite[Lemma 5.3.9]{cossonniere} and \cite[Corollary 5.1.4]{cossonniere}. Following the analysis in \cite[Lemma 5.3.6]{cossonniere} and \cite[Lemma 5.3.8]{cossonniere} one can show that $T(k):H^{-3/2}(\partial D)\times H^{1/2}(\partial D)\mapsto$ $H^{1/2}(\partial D)\times H^{-1/2}(\partial D)$ is a coercive operator. }Thus, the theory of eigenvalue problems for holomorphic Fredholm operator-valued functions applies to $Z(k)$.

\subsubsection{The numerical approximation of the system of boundary integral equations}
In this subsection, we briefly explain how we turn the system of boundary integral equations $Z(k)X=0$ given by (\ref{equation2}) into an algebraic system $\mathbf{Z}(k)v=0$, where $\mathbf{Z}(k)$ is a dense matrix of size $m\times m$ depending on the wavenumber $k$.
We use the boundary element collocation solver developed in \cite{kleefeldphd}, since we can obtain highly accurate approximations 
with a moderate size of $m$ which is due to superconvergence. Note that this solver has already been applied in several applications dealing with the Helmholtz equation (see for example \cite{kirschkleefeld,kleefeld1,kleefeld2,kleefeld3,kleefeldhabil,kleefeldlin1,kleefeldlin2}). We will use the parameters $\alpha=0.1$ for quadratic interpolation, $N_S=128$, and $N_{NS}=4$ with $768$ collocation points (see \cite{kleefeldphd,kleefeldlin2}). Hence, we have $m=1536$.

The resulting nonlinear eigenvalue problem $\mathbf{Z}(k)v=0$ with matrix $\mathbf{Z}(k)\in \mathbb{C}^{m\times m}$ with large $m$ is solved by the recent invented algorithm by Beyn (see \cite{beyn}). It uses Keldysh's theorem to reduce the nonlinear eigenvalue problem to a linear eigenvalue problem of much smaller size via complex-valued contour integrals which are numerically approximated by the trapezoidal rule. Precisely, one chooses a contour in the complex plane (usually a $2\pi$-periodic functions such as an ellipse) and calculates all nonlinear eigenvalues $k$ including their multiplicity situated inside the contour.

\subsubsection{The numerical calculation of the interior transmission eigenvalues} 
Now, we are in the position to present numerical results for various obstacles and different parameter settings. First, we choose a peanut-shaped obstacle which is parametrically given by the spherical coordinates $x=\varrho\sin(\phi)\cos(\theta)$, $y=\varrho\sin(\phi)\sin(\theta)$, and $z=\varrho\cos(\phi)$ with azimuthal angle $\phi\in [0,\pi]$ and polar angle $\theta\in [0,2\pi]$. The positive $\varrho$ is given by the equation $\varrho^2=9\left\{\cos^2(\phi)+\sin^2(\phi)/4\right\}/4$. We consider the two different index of refractions $n=1/2$ and $n=4$, since we have to distinguish from the theoretical point of view the cases $n<1$ and $n>1$. We pick $\eta=1/2$, $\eta=1$, and $\eta=3$. Hence, we have a total of six different cases under consideration.
The results are listed in Table \ref{tablepeanut}.
\begin{table}[!ht]
\centering
\begin{tabular}{l|cccccccccc}
    $(n,\eta)$      & 1. & 2. & 3. & 4. & 5. & 6. & 7.\\
 \hline
 (1/2,1/2) & 1.481\,359 & 1.754\,289 & 2.080\,586 & 2.106\,238 & 2.245\,421 & 2.436\,428& 2.524\,843\\
 (1/2,1)   & 1.889\,608 & 2.245\,548 & 2.713\,844 & 2.727\,860 & 2.934\,707 & 3.188\,784& 3.326\,731\\
 (1/2,3)   & 2.482\,082 & 2.947\,498 & 3.640\,550 & 3.695\,166 & 3.997\,475 & 4.327\,057& 4.599\,158\\
 \hline
 (4,1/2)   & 2.754\,035 & 2.987\,131 & 3.460\,241 & 3.517\,669 & 3.583\,455 & 3.777\,528& 3.809\,614\\
 (4,1)     & 2.678\,956 & 2.930\,558 & 3.404\,815 & 3.456\,156 & 3.534\,554 & 3.729\,411& 3.774\,500\\
 (4,3)     & 2.391\,812 & 2.723\,728 & 3.196\,562 & 3.198\,664 & 3.291\,749 & 3.560\,054& 3.640\,825\\
 \hline
\end{tabular}
\caption{\label{tablepeanut}The interior transmission eigenvalues for a peanut-shaped obstacle using the index of refractions $n=1/2$ and $n=4$ for $\eta=1/2$, $\eta=1$, and $\eta=3$.}
\end{table}
As we can see, we are able to compute various interior transmission eigenvalues to high accuracy for a peanut-shaped obstacle. For comparison purpose we also list the first seven classic interior transmission eigenvalues using $n=4$. They are $2.825456$, $3.044765$, $3.515130$, $3.574902$, $3.627453$, $3.827094$, and $3.844736$ (see \cite[Table 4]{kleefeldITP}).
Next, we show for a fixed $n=4$ the first interior transmission eigenvalue for various choices of $\eta$ which illustrates the monotonicity. As we can see in Figure \ref{fig1}, the first interior transmission eigenvalue decreases as $\eta$ increases. The same is true for the other interior transmission eigenvalues.
\begin{figure}[!ht]
\centering
\epsfig{file=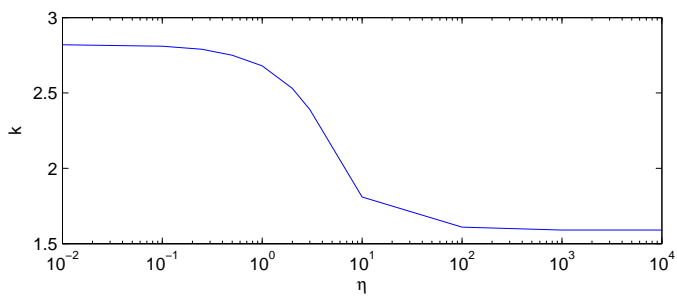,height=4cm}
\caption{\label{fig1}The monotonicity of the first interior transmission eigenvalues for the peanut-shaped obstacle using $n=4$ for increasing $\eta$.}
\end{figure} 
 
We also show the monotonicity of the first interior transmission eigenvalue for the peanut-shaped obstacle using $n=1/2$ for increasing $\eta$ in Figure \ref{fig1b}.

\begin{figure}[!ht]
\centering
\epsfig{file=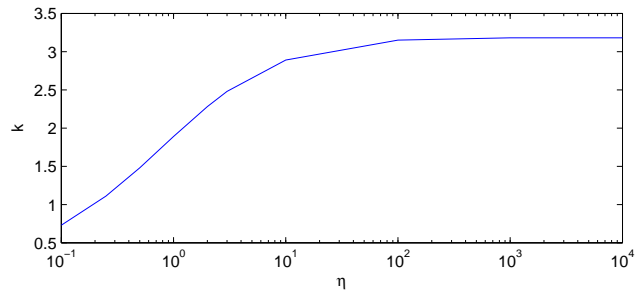,height=4cm}
\caption{\label{fig1b}The monotonicity of the first interior transmission eigenvalues for the peanut-shaped obstacle using $n=1/2$ for increasing $\eta$.}
\end{figure} 
Additionally, we compute the interior transmission eigenvalues for a cushion-shaped object that is given parametrically by spherical coordinates with $\varrho=1-\cos(2\phi)/2$. We again consider the same parameters as in the previous example. The results are shown in Table \ref{tablecushion}.
\begin{table}[!ht]
\centering
\begin{tabular}{l|cccccccccc}
  $(n,\eta)$        & 1. & 2. & 3. & 4. & 5. & 6. & 7.\\
 \hline
 (1/2,1/2) & 1.359\,283 & 1.694\,494 & 2.012\,440 & 2.087\,716 & 2.110\,396 & 2.271\,313& 2.298\,900\\
 (1/2,1)   & 1.730\,859 & 2.164\,577 & 2.595\,767 & 2.732\,528 & 2.979\,526 & 2.994\,836& 3.235\,653\\
 (1/2,3)   & 2.273\,696 & 2.834\,967 & 3.439\,393 & 3.651\,267 & 3.766\,782 & 4.031\,330& 4.100\,045\\
 \hline
 (4,1/2)   & 2.863\,595 & 2.878\,783 & 3.144\,915 & 3.159\,434 & 3.469\,001 & 3.814\,417& 3.828\,743\\
 (4,1)     & 2.762\,018 & 2.818\,074 & 3.087\,199 & 3.099\,157 & 3.431\,516 & 3.763\,499& 3.782\,309\\
 (4,3)     & 2.384\,383 & 2.611\,343 & 2.841\,059 & 2.945\,477 & 3.305\,505 & 3.508\,923& 3.583\,133\\
 \hline
\end{tabular}
\caption{\label{tablecushion}The interior transmission eigenvalues for a cushion-shaped obstacle using the index of refractions $n=1/2$ and $n=4$ for $\eta=1/2$, $\eta=1$, and $\eta=3$.}
\end{table}
As before, we are able to compute various accurate interior transmission eigenvalues for a cushion-shaped obstacle. The first seven classic interior transmission eigenvalues are $2.941084$, $2.962924$, $3.192652$, $3.234727$, $3.508462$, $3.848378$, and $3.892142$ (see \cite[Section 6.5]{kleefeldITP}). Note that we are also able to compute interior transmission eigenvalues for other obstacles as well, but it should be enough to provide them for a sphere, a peanut, and a cushion. 

\subsubsection{Computing interior transmission eigenvalues from far-field data}
In this section, we give some numerical examples to show that the real transmission eigenvalues corresponding to our problem can be determined from the far-field measurements following the approach in \cite{cchlsm} (also see \cite{cavityhcs}). To this end, we introduce the far-field equation. Let the radiating fundamental solution to Helmholtz equation in $\R^3$ be denoted by $\Phi_k(x,z)$, $x,z\in \R^3, x\neq z$ with $\Phi_{\infty}(\hat{x}, z)=\frac{1}{4\pi}\mathrm{e}^{\mathrm{i} k \hat{x} \cdot z}$, $z\in \R^3$ is fixed, denoting the far-field pattern in the direction $\hat{x} \in \mathbb{S}$. Let $u^s$ be the scattered field corresponding to the scattering problem \eqref{direct1}--\eqref{src} and therefore having the expansion (see for example \cite{p1})
\[
u^s(x,d)=\frac{\mathrm{e}^{\mathrm{i}k|x|}}{|x|} \left\{u_{\infty}(\hat{x}, d ) + \mathcal{O} \left( \frac{1}{|x|}\right) \right\}\; \textrm{  as  } \;  |x| \to \infty. 
\]
Now, define the corresponding far-field operator $\mathcal{F}: L^2(\mathbb{S}) \mapsto  L^2(\mathbb{S})$
$$(\mathcal{F}g)(\hat{x})=\int_{\mathbb{S}} u_\infty(\hat{x},d) g(d) \, \dif s(d).$$
We now define the far-field equation  
\begin{eqnarray}
&&\textrm{for a fixed  } z \in D \textrm{  and  } |\hat{x}|=1 \,,\, \, \textrm{ find } \, \,  g_z \in L^2(\mathbb{S}) \, \,\notag\\
&& \hspace{0.5in} \textrm{ such that } \, \,  (\mathcal{F}g_z)(\hat{x})=\Phi_{\infty}(\hat{x},z).
\label{farfieldeq}
\end{eqnarray}
Since $\mathcal{F}$ is compact we find the Tikhonov regularized solution, $g_{z,\delta}$ of the far-field equation defined as the unique minimizer of 
$$\| \mathcal{F} g-\Phi_{\infty}(\cdot, z)\|^2_{L^2(\mathbb{S})}+\epsilon\|g\|^2_{L^2(\mathbb{S})},$$
where the regularization parameter $\epsilon:=\epsilon(\delta)\to 0$ as the noise level $\delta\to 0$. The regularization parameter is chosen based on Morozov's discrepancy principle.
At a transmission eigenvalue we expect
$||{g_{z,\delta}}||_{L^2(\mathbb{S})} \rightarrow \infty$ as $\delta\to 0$. Therefore, the transmission eigenvalues should appear as spikes in the plot of $k \mapsto ||{g_{z}}||_{L^2(\mathbb{S})}$. Below we present an example of 
computing the first transmission eigenvalue from the far field measurements for a unit sphere.

We first have to compute the far-field pattern in order to approximate the far-field operator. The derivation of the far-field pattern for a sphere of radius $R$ 
is an easy task. Using the same ansatz as in \cite[Section 4.2]{anachakle}, we obtain
\begin{eqnarray*}
&&\left[
\begin{array}{rr}
	h_p^{(1)}(k R) & \quad  -j_p(k\sqrt{n} R)\\
	k h_p^{(1)'}(k R)+\eta h_p^{(1)}(k R)  & \quad  - k\sqrt{n} j_p'(k\sqrt{n} R)
\end{array}
\right]\left[
\begin{array}{c}
	\alpha_p^m\\
	\beta_p^m
\end{array}
\right]\\
&&\hspace{1.4in} =-4\pi \mathrm{i}^p \overline{Y_p^m({d})}\left[
\begin{array}{l}
	j_p(k R)\\
	j_p'(k R)k+\eta j_p(k R)
\end{array}
\right]\,,
\end{eqnarray*}
where $j_p$ denotes the spherical Bessel
function of the first kind of order $p$, $h_p^{(1)}$ the spherical Hankel
function of the first kind of order $p$, $P_p$ is the Legendre polynomial of
order $p$, and $Y_p^m$ is the spherical harmonic.
The solution $\alpha_p^m$ is given by
\begin{eqnarray*}
\alpha_p^m=-4\pi \mathrm{i}^p \overline{Y_p^m({d})} \frac{k\sqrt{n} j_p'(k\sqrt{n} R)j_p(k R)-j_p(k\sqrt{n} R)\left(k j_p'(k R)+\eta j_p(k R) \right)}{k\sqrt{n} j_p'(k\sqrt{n} R)h_p^{(1)}(k R)- j_p(k\sqrt{n} R)\left(k h_p^{(1)'}(k R)+\eta h_p^{(1)}(k R)\right)}\,.
\end{eqnarray*}
The far-field pattern of the scattering wave is given by (see \cite[Theorem 2.16]{coltonkress})
\begin{eqnarray*}
&&u^\infty(\hat{x}, {d})\\
&=&\frac{1}{k}\sum_{p=0}^\infty \frac{1}{\mathrm{i}^{p+1}}\sum_{m=-p}^m \alpha_p^m Y_p^m(\hat{x})\\
&=&\frac{\mathrm{i}}{k}\sum_{p=0}^\infty (2p+1) \frac{k\sqrt{n} j_p'(k\sqrt{n} R)j_p(k R)-j_p(k\sqrt{n} R)\left(k j_p'(k R)+\eta j_p(k R) \right)}{k\sqrt{n} j_p'(k\sqrt{n} R)h_p^{(1)}(k R)- j_p(k\sqrt{n} R)\left(k h_p^{(1)'}(k R)+\eta h_p^{(1)}(k R)\right)}
P_p(\hat{x} \cdotp {d})\,.
\end{eqnarray*}
In the last step we used the addition theorem for spherical harmonics. To derive the far-field pattern for an arbitrary obstacle, we use the ansatz of a combination of a double and single layer potential as in \cite[Section 4.1]{anachakle} to avoid complicated right-hand side functions. We then obtain the $2\times 2$ system of boundary integral equations
\begin{eqnarray*}
&&\Bigg(\left[
\begin{array}{cc}
	I      & \quad  0\\
	0      &\quad -I
\end{array}
\right]+\left[
\begin{array}{cc}
	D_{k}-D_{k\sqrt{n}} &\quad  S_{k}-S_{k\sqrt{n}}\\
	T_{k}-T_{k\sqrt{n}}+\eta\left(D_{k}+\frac{1}{2}I\right) &\quad  D_{k}-D_{k\sqrt{n}}+\eta S_{k}
\end{array}
\right]\Bigg)\left[
\begin{array}{c}
	\phi\\
	\psi
\end{array}
\right]\\
&&\hspace{1in}=\left[
\begin{array}{c}
 -u^{\text{inc}}\\
 -\partial_\nu u^{\text{inc}}-\eta u^{\text{inc}}
\end{array}
\right]\,,
\label{system}
\end{eqnarray*} 
where $\phi$ and $\psi$ are two unknown density functions which depend on the direction of incident ${d}$. The far-field pattern is given by
\begin{eqnarray*}
u^{\infty}(\hat{x},{d}) = 
\frac{1}{4\pi}\int_{\partial D}\left[\partial_{\nu(y)} \mathrm{e}^{-\mathrm{i}k \hat{x} \cdotp y} \phi(y;{d})+\mathrm{e}^{-\mathrm{i}k\hat{x} \cdotp y}\psi(y;{d})\right]\;\mathrm{d}s(y),
\qquad \hat{x} \in\mathbb{S}\,.
\end{eqnarray*}
Approximating the far-field equation (\ref{farfieldeq}) and using the origin as a sampling point gives us the possibility of detecting at least the first interior transmission eigenvalue for the unit sphere and the peanut-shaped object as shown in Figure \ref{figLSM}, where we used the parameters $n=4$ and $\eta=1$. The chosen grid is $[2.5,4.5]$ with grid size $0.01$.
\begin{figure}[H]
\centering
\subfigure[The unit sphere.]{
\epsfig{file=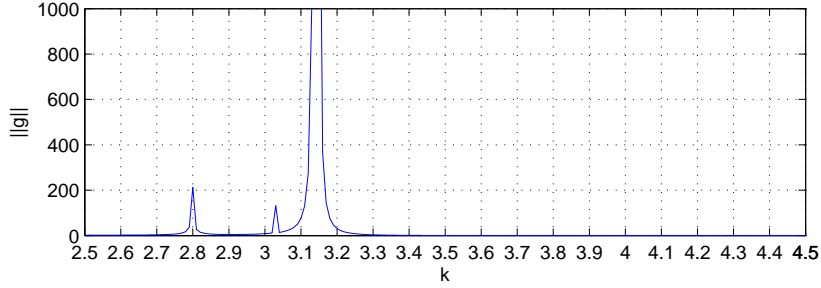,height=4cm}
}
\subfigure[The peanut-shaped obstacle.]{
\epsfig{file=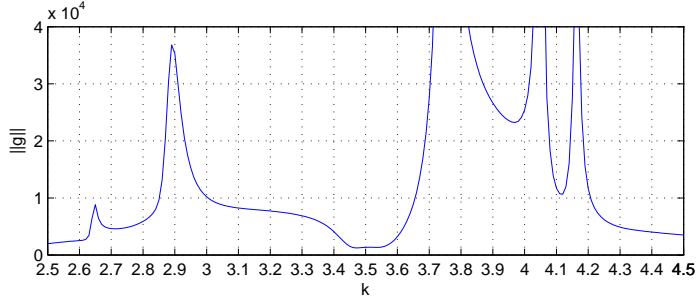,height=4cm}
}
\caption{\label{figLSM}The detection of the first interior transmission eigenvalue via the linear sampling method for the unit sphere and the peanut-shaped obstacle using $n=4$ and $\eta=1$.}
\end{figure}

As we can see, we are able to detect the first three interior transmission eigenvalues for the unit sphere. We obtain the values $2.80$, $3.03$, and $3.14$. The highly accurate values are given by $2.798\,386$, $3.029\,807$, and $3.141\,593$, respectively. Hence, we obtain accuracy within the chosen grid size. The situation slightly changes for the peanut-shaped obstacle. We are able to detect the first two interior transmission eigenvalues within a reasonable accuracy. We get $2.65$ and $2.89$, whereas the highly accurate values are $2.678\,956$ and $2.930\,558$, respectively. The values are accurate within two digits. Note that we are able to detect more interior transmission eigenvalues. Precisely, we obtain the values $3.75$, $4.05$, and $4.16$. The value $3.75$ is the seventh interior transmission eigenvalue as shown in Table \ref{tablepeanut}. Note that the theoretical validation of this approach would be future research.

\section{Summary}
In this article, we have extended the classical interior transmission problem to an interior transmission eigenvalue problem with the boundary condition of the form
\[
v = w \quad \text{ and } \quad \frac{\pa w}{\pa \nu} - \frac{\pa v}{\pa \nu} = \eta v \quad \text{ on } \pa D,
\]
for an inhomogeneous media. We proved existence and discreteness of the interior transmission eigenvalues and investigated the inverse spectral problem. Additionally, we proved monotonicity
with respect to the refractive index as well as the boundary conductivity parameter for the first transmission eigenvalues.
Further, we showed a uniqueness result for constant coefficients. All theoretical results are confirmed with numerical results. The possibility to calculate interior transmission eigenvalues from far-field data seems possible, but has to be analyzed from the theoretical point of view, which is left as a future research.

\section*{Acknowledgment}
Some of the results of this work were carried out during a research stay of Oleksandr Bondarenko at the University of Delaware in Summer 2015. He greatly acknowledges the hospitality of Fioralba Cakoni during the stay and the 
financial support from the Karlsruhe House of Young Scientists (KHYS). The authors would like to thank Fioralba Cakoni for valuable advice and the fruitful discussions.


\end{document}